\documentclass[11pt]{amsart}
\usepackage[T1]{fontenc}
\usepackage[utf8]{inputenc}
\usepackage[french]{babel} 
\usepackage{geometry}                
\usepackage{graphicx}
\usepackage{amsmath, amsthm, amssymb, amscd}
\usepackage{mathrsfs}
\usepackage{hyperref}
\usepackage[all]{xy}
\usepackage{tikz}

\theoremstyle{plain}
\newtheorem{thm}{Théorème}[section]
\newtheorem{lem}[thm]{Lemme}
\newtheorem{prop}[thm]{Proposition}
\newtheorem{cor}[thm]{Corollaire}

\newtheorem{lemdef}[thm]{Lemme-Définition}

\theoremstyle{definition}
\newtheorem{defn}{Définition}[section]
\newtheorem{conj}{Conjecture}[section]
\newtheorem{example}{Exemple}[section]

\theoremstyle{remark}
\newtheorem{rem}{Remarque}[section]

\newcommand{\gl}{\mathrm{GL}}
\newcommand{\ggl}{\mathfrak{gl}}
\newcommand{\ks}{\mathfrak{S}}
\newcommand{\pgl}{\mathrm{PGL}}
\newcommand{\bbg}{\mathbb{G}}

\newcommand{\cf}{\mathcal{F}}
\newcommand{\co}{\mathcal{O}}
\newcommand{\ckk}{\mathcal{K}}
\newcommand{\cl}{\mathcal{L}}

\newcommand{\cvv}{\mathcal{V}}

\newcommand{\xx}{\mathscr{X}}
\newcommand{\mt}{\widetilde{T}}

\newcommand{\kg}{\mathfrak{g}}
\newcommand{\kn}{\mathfrak{n}}
\newcommand{\kp}{\mathfrak{p}}
\newcommand{\kt}{\mathfrak{t}}

\newcommand{\ku}{\mathfrak{u}}
\newcommand{\km}{\mathfrak{m}}

\newcommand{\bn}{\mathbf{N}}
\newcommand{\bz}{\mathbf{Z}}
\newcommand{\br}{\mathbf{R}}

\newcommand{\bq}{\mathbf{Q}}
\newcommand{\ba}{\mathbf{a}}
\newcommand{\bg}{\mathbf{G}}

\newcommand{\baa}{\mathbf{A}}

\newcommand{\val}{\mathrm{val}}
\newcommand{\lie}{\mathrm{Lie}}
\newcommand{\Ad}{\mathrm{Ad}}
\newcommand{\ad}{\mathrm{ad}}
\newcommand{\tr}{\mathrm{Tr}}

\newcommand{\diag}{\mathrm{diag}}

\newcommand{\sch}{\mathrm{Sch}}

\newcommand{\vep}{\epsilon}
\newcommand{\tlambda}{\tilde{\lambda}}

\author{Zongbin \textsc{Chen}}

\address{Département de mathématiques, Bât. 425\\ Université Paris-sud 11, 91405 Orsay-Cedex, France} 
\email{zongbin.chen@math.u-psud.fr} 
\title{Pureté des fibres de Springer affines pour $\mathrm{GL}_{4}$} 

\begin{document}

\maketitle

\begin{abstract}

Pour $\gl_{4}$ et $\gamma\in \ggl_{4}(F)$ un élément semi-simple régulier non-ramifié entier, la fibre de Springer affine $\xx_{\gamma}$ admet un pavage en espaces affines, donc sa cohomologie est ``pure''.

\end{abstract}

\section{Introduction}

Soit $k$ un corps algébriquement clos. On note $F=k((\vep))$ le corps de séries de Laurent sur $k$, $\co=k[[\vep]]$ son anneau d'entier, $\kp=\vep k[[\vep]]$ son idéal maximal. On fixe une clôture algébrique $\overline{F}$ de $F$, et $\val: \overline{F}\to \bq$ la valuation discrete normalisée par $\val(\vep)=1$. Soient $G=\gl_{d}$, $T$ le tore maximal des matrices diagonales, $B$ le sous-groupe de Borel des matrices triangulaires supérieures de $G$. On notera leur algèbre de Lie par la lettre gothique correspondante. Soient $K=G(\co)$, $I$ le sous-groupe d'Iwahori standard de $G(F)$, i.e. il est l'image inverse de $B$ sous la réduction $G(\co)\to G(k)$. Les groupes $G(F),\,K,\,I$ sont muni des structures de ind-$k$-schéma en groupe. On note $\xx=G(F)/K$ la grassmannienne affine, c'est un ind-$k$-schéma qui classifie les réseaux dans $F^{d}$:
$$
\xx=\{L\subset F^{d}\mid L \text{ est un } \co\text{-module de type fini tel que } L\otimes F=F^{d}\}.
$$

Soit $\gamma\in \kg(F)$ un élément semi-simple régulier. La fibre de Springer affine 
\begin{eqnarray*}
\xx_{\gamma}=\{g\in G(F)/K\,\mid\,\Ad(g^{-1})\gamma \in \kg(\co)\}\end{eqnarray*}
a été introduite par Kazhdan et Lusztig dans \cite{kl}. C'est un sous-schéma fermé localement de type fini et de dimension finie de $\xx$, qui est non-vide si et seulement si $\gamma$ est \emph{entier} (i.e. ses valeurs propres sont entières dans $\overline{F}$). Elle est utilisée par Goresky, Kottwitz et Macpherson dans \cite{gkm2} pour montrer le lemme fondamental de Langlands-Shelstad, sous l'hypothèse suivante:

\begin{conj}[Goresky-Kottwitz-Macpherson]Soit $\gamma\in \kg(F)$ un élément semi-simple régulier entier, la cohomologie de la fibre de Springer affine $\xx_{\gamma}$ est pure au sens de Grothendieck-Deligne.\end{conj}

Dans \cite{gkm1}, Goresky, Kottwitz et Macpherson ont montré cette conjecture pour $\gamma$ équivalué. L'élément $\gamma$ est dit \emph{équivalué} si $\val(\alpha(\gamma))$ ne dépend pas de la racine $\alpha$ de $G$ sur $\overline{F}$ par rapport à $Z_{\gamma}(G)$. Pour cela, ils ont construit un pavage en espaces affines de $\xx_{\gamma}$. Pour une variété $X$ sur $k$, un \emph{pavage en espaces affines} de $X$ est une filtration croissante exhaustive $X_{0}\subset X_{1}\subset\cdots$ de $X$ telle que $X_{i}$ est fermé et $X_{i}\backslash X_{i-1}$ est isomorphe à un espace affine standard, $\forall i$. Dans le cas où $\gamma\in \kt(\co)$ est équivalué, un tel pavage est obtenu en intersectant $\xx_{\gamma}$ avec le pavage de Bruhat-Tits.

Mais pour $\gamma\in \kt(\co)$ non-équivalué, les intersections $\xx_{\gamma}\cap IvK/K$ sont en général singulières. Un exemple typique est le suivant.  

\begin{example}\label{typique}

Soit $G=\gl_{3}$, $\gamma=\begin{pmatrix}\vep^{2}&&\\ &\vep^{4}& \\
&& -\vep^{4}\end{pmatrix}$. Pour $v\in \bz^{3}$, on note $\vep^{v}=\begin{pmatrix}\vep^{v_{1}}&&\\ &\vep^{v_{2}}& \\ &&\vep^{v_{3}}\end{pmatrix}$ et $C(v)=I\vep^{v}K/K$. On a
$$
C(0,2,-2)=\begin{bmatrix}1&&\co/\kp^{2}\\ \kp/\kp^{2} &1& \co/\kp^{4} \\
&& 1\end{bmatrix}\vep^{v}K/K.
$$
En utilisant la coordonné 
$$
\begin{bmatrix}1&&a_{0}+a_{1}\vep \\ b_{1}\vep &1& \sum_{i=0}^{3}c_{i}\vep^{i} \\
&& 1\end{bmatrix}\in \begin{bmatrix}1&&\co/\kp^{2}\\ \kp/\kp^{2} &1& \co/\kp^{4} \\
&& 1\end{bmatrix},
$$ 
on trouve que l'intersection $\xx_{\gamma}\cap IvK/K$ est la sous-variété de $\baa^{7}$ définie par l'équation
$$
a_{0}b_{1}=0.
$$

Lucarelli a construit dans \cite{v} un pavage en espaces affines de $\xx_{\gamma}$ pour $\pgl_{3}$. Il part d'un pavage en espaces affines de $\xx$ qui est différent de celui de Bruhat-Tits. Dans notre exemple pour $\gl_{3}$, Lucarelli rassemble le pavé singulier $C(0,2,-2)\cap \xx_{\gamma}$ et le pavé lisse $C(1,1,-2)\cap \xx_{\gamma}$, et redécoupe la réunion de ces deux pavés en utilisant la décomposition de Bruhat-Tits pour l'Iwahori 
$$
I'=\Ad(\diag(1,\vep^{2},\vep^{2}))I.
$$ 
Le pavé singulier $C(0,2,-2)\cap \xx_{\gamma}$ sera coupé en $2$ parties. D'une part on a la branche $b_{1}=0$, qui est isomorphe à $\baa^{6}$, d'autre part, on a la branche $b_{1}\neq 0,\,a_{0}=0$, qui est isomorphe à $\bbg_{m}\times \baa^{5}$. Cette dernière sera réunit avec la cellule $C(1,1,-2)\cap \xx_{\gamma}$ pour former l'espace affine $\baa^{6}$, ce que on peut voir dans le calcul suivant:

\begin{eqnarray*}
&&\begin{bmatrix}1&&a_{1}\vep \\ b_{1}\vep &1& \sum_{i=0}^{3}c_{i}\vep^{i} \\
&& 1\end{bmatrix}\begin{bmatrix}1&&\\ &\vep^{2}&\\ &&\vep^{-2}
\end{bmatrix}K/K \\
&&=\begin{bmatrix}1&b_{1}^{-1}\vep^{-1}&a_{1}\vep-b_{1}^{-1}c_{3}\vep^{2} \\  &1& \sum_{i=0}^{2}c_{i}\vep^{i} \\
&& 1\end{bmatrix}\begin{bmatrix}\vep&&\\ &\vep&\\ &&\vep^{-2}
\end{bmatrix}K/K,\end{eqnarray*}
et
$$
C(1,1,-2)\cap \xx_{\gamma}=\begin{bmatrix}1&&\kp/\kp^{3}\\  &1& \co/\kp^{3} \\
&& 1\end{bmatrix}\begin{bmatrix}\vep&&\\ &\vep&\\ &&\vep^{-2}
\end{bmatrix}K/K.
$$
Donc la branche $b_{1}\neq 0,\,a_{0}=0$ et la cellule $C(1,1,-2)\cap \xx_{\gamma}$ se rassemblent en l'espace affine 
$$
\begin{bmatrix}1&\kp^{-1}/\co&\kp/\kp^{3}\\  &1& \co/\kp^{3} \\
&& 1\end{bmatrix}\begin{bmatrix}\vep&&\\ &\vep&\\ &&\vep^{-2}
\end{bmatrix}K/K\cong \baa^{6}.
$$

\end{example}

Remarquons que dans l'exemple \ref{typique}, on peut aussi déplacer la branche $a_{0}\neq 0,\,b_{1}=0$ vers le pavé lisse $C(-2,2,0)\cap \xx_{\gamma}$, ce que ne fait pas Lucarelli. Nous utilisons en fait les deux possibilités pour obtenir une famille de  pavages de $\gl_{3}$ qui sont différents de celui de Lucarelli, mais qui nous permet de paver la fibre de Springer affine pour $\gl_{4}$.

Soit $\{e_{i}\}_{i=1}^{d}$ la base standard de $F^{d}$. Pour $m\in \bz$, on note
$$
\xx_{\geq -m}=\{L\in \xx\,\mid \, L\subset \vep^{-m}\co^{d}\}.
$$
C'est un sous schéma fermé $T$-invariant de $\xx$, et $\xx=\lim_{m\to +\infty}\xx_{\geq -m}$. Notre résultat principal est

\begin{thm}
Pour $G=\gl_{4}$ et $\gamma\in \kg(F)$ semi-simple régulier non-ramifié entier, $\xx_{\geq -m}\cap \xx_{\gamma}$ admet un pavage en espaces affines, en particulier, $\xx_{\gamma}$ est pur.
\end{thm}

L'idée est de couper $\xx_{\geq -m}\cap \xx_{\gamma}$ en parties localement fermées telles que chaque partie est une fibration en espaces affines sur une sous-variété localement fermée de $\xx_{\gamma}^{\gl_{3}}$, et donc le pavage est ramené aux pavages pour $\gl_{3}$.\\

\textbf{Notations.} On note $\Phi(G,T)=\{\alpha_{i,j}\}$ le système de racines de $G$ par rapport à $T$ et on simplifie $\alpha_{i}=\alpha_{i,i+1}$, on note $\{\varpi_{i}\}_{i=1}^{d-1}$ les poids fondamentaux correspondants. À toute racine $\alpha\in \Phi(G,T)$, on associe de la manière usuelle une co-racine $\alpha^{\vee}\in X_{*}(T)$. On note $X_{*}^{+}(T)$ le semi-groupe des co-caractères dominants. On note $W=\ks_{d}$ le groupe de Weyl de $G$, on note $s_{i,j}\in W$ la réflexion associée à la racine $\alpha_{i,j}$ et on simplifie $s_{i}=s_{i,i+1}$. Pour tout sous-groupe fermé $H$ de $G$ stable par $T$ sous l’action adjointe, on note $\Phi(H,T)$ l’ensemble des racines de $T$ dans $\lie(H)$. On note $\cf(T)$ l'ensemble des sous-groupes paraboliques de $G$ contenant $T$, et $\cl(T)$ l'ensemble des sous-groupes de Levi contenant $T$.

Pour $M\in \cl(T)$, on utilise un exposant $^{M}$ pour désigner l'objet correspondant pour $M$. On identifie la grassmannienne affine $\xx^{M}$ à un sous-ind-$k$-schéma fermé de $\xx$ par l'injection naturel $mM(\co)\to mK,\,\forall  m\in M(F)$. 

Pour $x\in \br$, on note $\lfloor x \rfloor$ le plus grand entier qui est inférieure ou égale à $x$, et $\lceil x \rceil$ le plus petit entier qui est supérieur ou égale à $x$.\\

\textbf{Remerciements.} C'est un grand plaisir pour moi de remercier G. Laumon pour m'avoir proposé ce sujet de recherche et ses encouragements constants. Je le remercie aussi pour les nombreuses améliorations qu'il a apporté à ce travail. Je remercie U. G\"ortz et T. Haines pour avoir signalé quelques imprécisions sur la longueur et l'ordre de Bruhat-Tits sur le groupe de Weyl affine. Je remercie enfin le rapporteur anonyme de cet article pour sa relecture attentive.

\section{Pavages non standard de la grassmannienne affine}

\subsection{Filtration de Moy-Prasad}

Le tore ``pivotant'' $\bbg_{m}$ agit sur le corps $F=k((\vep))$ par $t*\vep^{n}=t^{n}\vep^{n},\,\forall t\in k^{\times}, n\in \bz$. Ainsi il agit sur $G(F)$ et son algèbre de Lie $\kg_{F}$. Soit $\widetilde{T}=\bbg_{m}\times T$, le premier facteur étant le tore pivotant. On note $\nu_{0}\in X^{*}(\bbg_{m})$ le caractère définit par $\nu_{0}(t)=t$. On note $(n,\alpha_{i,j})$ le caractère $(\nu_{0}^{n},\,\alpha_{i,j})$ de $\widetilde{T}$.

L'algèbre de Lie $\kg_{F}$ se décompose en espaces propres sous l'action de $\widetilde{T}$:
$$
\kg_{F}=\bigoplus_{m\in \bz}\vep^{m}\kt\oplus\bigoplus_{(n,\alpha)\in \bz\times\Phi(G,T)}\kg_{\alpha}\vep^{n}+\kg\vep^{N},\quad  N\gg 0,
$$
où $\vep^{m}\kt$ est de poids $(m,\,0)$ et $\kg_{\alpha}\vep^{n}$ est de poids $(n,\alpha)$.

Pour $x\in \kt$ fixé, $t\in \br$, on définit une filtration sur $\kg_{F}$:
$$
\kg_{x,t}=\bigoplus_{\substack{(n,\alpha)\in \bz\times\Phi(G,T)\\\alpha(x)+n\geq t}}\kg_{\alpha}\vep^{n}+\kg\vep^{N},\quad N\gg 0.
$$ 
C'est la filtration de Moy-Prasad sur $\kg_{F}$ introduite dans \cite{mp}. Pour $t\geq 0$, on note $\bg_{x,t}$ le sous-groupe de $G(F)$ contenant $T$ dont l'algèbre de Lie est $\kg_{x,t}$. Alors, $\bg_{x}:=\bg_{x,0}$ est un sous-groupe parahorique de $G(F)$ contenant $T$, et $\bg_{x,t}$ est un sous-groupe distingué de $\bg_{x}$.

\begin{example}\label{notationx}

\begin{enumerate}

\item Pour $0\in \kt$, on a $\bg_{0}=K$. 

\item Pour $x_{0}=\left(\frac{d}{d},\frac{d-1}{d},\cdots,\frac{1}{d}\right)\in \kt$, on a $\bg_{x_{0}}=I$. 

\item Pour $\ba\in \bz^{d}$, on note $x_{\ba}=\left(\frac{d}{d}-a_{1},\frac{d-1}{d}-a_{2},\cdots,\frac{1}{d}-a_{d}\right)\in \kt$, on a
$$
\bg_{x_{\ba}}=I_{\ba}:=\Ad(\vep^{\ba})I.
$$ 
\end{enumerate}
\end{example}

\subsection{La décomposition de Bruhat-Tits}

La grassmannienne affine admet un pavage standard en espaces affines. Pour le décrire, on définit d'abord une ordre de Bruhat-Tits ``modifiée'' $\prec_{I_\ba}$ sur $X_{*}(T)$, on commence par $\prec_{I}$. Soit $v,v'\in X_{*}^{+}(T)$, alors $v\prec_{I}v'$ si et seulement si
\begin{eqnarray*}
v_{1}&\leq &v'_{1};\\
v_{1}+v_{2}&\leq & v'_{1}+v'_{2};\\
&\vdots&\\
v_{1}+\cdots+v_{d}&=&v'_{1}+\cdots+v'_{d}.
\end{eqnarray*}
Puis on pose $Wv\prec_{I}Wv'$. Pour tout $g,g'\in W/W_{v}$, où $W_{v}$ est le stabilisateur de $v$, on pose $gv\prec_{I} g'v$ si et seulement si $g'\prec_{B}g$, où $\prec_{B}$ est l'ordre sur $W/W_{v}$ induite de celle de Bruhat-Tits sur $W$ par rapport à $B$. Puis, pour $v,v'\in X_{*}(T)$, on pose
$$
v\prec_{I_{\ba}}v'\iff \vep^{-\ba}v\prec_{I}\vep^{-\ba}v'.
$$

On identifie $X_{*}(T)$ avec $\xx^{T}=T(F)/T(\co)$ par l'application $v\to \vep^{v}$.

\begin{thm}[Bruhat-Tits]

Pour $\ba\in \bz^{d}$, on a un pavage en espaces affines
$$
\xx=\bigsqcup_{v\in  X_{*}(T)}I_{\ba}vK/K,
$$
De plus, $I_{\ba}v'K/K\subset \overline{I_{\ba}vK/K}$ si et seulement si $v'\prec_{I_{\ba}} v$.
\end{thm}

On peut consulter \cite{iwahori} pour la démonstration. On va réécrire le théorème ci-dessus sous la forme d'une décomposition de Bialynicki-Birula. Le tore $\widetilde{T}$ agit sur $\xx$. On note $\tlambda_{\ba}\in X_{*}(\mt)$ le co-caractère défini par 
\begin{equation}\label{tlam}
\tlambda_{\ba}(t)=(t^{d},\,\diag(t^{d-a_{1}d},\, t^{d-1-a_{2}d},\,\cdots,\,t^{1-a_{d}d})).
\end{equation}
Considérons l'action de $\bbg_{m}$ sur $\xx$ induite par le co-caractère $\tlambda_{\ba}\in X_{*}(\mt)$. L'ensemble des points fixes $\xx^{\bbg_{m}}$ est discret et égal à $\{vK,\;v\in X_{*}(T)\}$. De plus, on a
$$
I_{\ba}vK/K=\{L\in \xx\,\mid\,\lim_{t\to 0}\tlambda_{\ba}(t)L=vK\}.
$$
Ici, la limite porte le sens suivant: Soit $N\in \bn$ assez grand tel que $L\in \xx_{\geq -N}$, alors le morphisme $\lambda_{L}:\bbg_{m}\to \xx$ défini par $\lambda_{L}(t)=\tlambda_{\ba}(t)L,\, \forall t\in k^{\times}$ se factorise par $\xx_{\geq -N}$ puisque $\xx_{\geq -N}$ est stable sous l'action de $\widetilde{T}$. Et il se prolonge à un morphisme unique $\bar{\lambda}_{L}:\baa^{1}\to \xx_{\geq -N}$ puisque $\xx_{\geq -N}$ est propre, la limit en question est définie comme
$$
\lim_{t\to 0}\tlambda_{\ba}(t)L=\bar{\lambda}_{L}(0),
$$
qui ne dépend que de $\tlambda_{\ba}$ et de $L$.

\subsection{Pavages non standard de la grassmannienne affine tronqu\'ee}

\subsubsection{Pavage triangulaire}

On va paver la variété $\xx_{\geq m}$ d'une manière différente de celle de Bruhat-Tits.

\begin{prop}\label{pav1}

Soit $\ba=(a_{1},\cdots,a_{d})\in \bz^{d}$  et $w\in X_{*}(T)$. 

\begin{enumerate}

\item L'intersection $\xx_{\geq m}\cap I_{\ba}wK/K$ est non-vide seulement si $w\in \xx_{\geq m}$.

\item Pour $w\in \xx_{\geq m}$, l'intersection $\xx_{\geq m}\cap I_{\ba}wK/K$ est isomorphe à un espace affine standard. Plus précisément, $\xx_{\geq m}\cap I_{\ba}wK/K=J_{\ba,m,w}wK/K$, où $J_{\ba,m,w}$ est la sous-$k$-variété ouverte et fermée de $I_{\ba}$ formée des matrices $(x_{i,j})$ 
telles que $x_{i,i}=1$ et que
$$
\val(x_{i,j})\geq m_{i,j},\quad\forall i\neq j,
$$ 
où $m_{i,j}=\max(a_{i}-a_{j}+\frac{i-j}{d},\,m-w_{j})$.

\item Par conséquent, on a un pavage en espaces affines
$$
\xx_{\geq m}=\bigsqcup_{w\in \xx_{\geq m}^{T}}\xx_{\geq m}\cap I_{\ba}wK/K.
$$
L'inclusion $(\xx_{\geq m}\cap I_{\ba}vK/K)\subset\overline{(\xx_{\geq m}\cap I_{\ba}wK/K)}$ implique que $v\prec_{I_{a}} w$. 
\end{enumerate}\end{prop}

\begin{proof}

\begin{enumerate}

\item  Prenons $y\in\xx_{\geq m}\cap I_{\ba}wK/K$. Puisque $\xx_{\geq m}$ est fermé et invariant sous l'action de $\mt$, on a 
$$
w=\lim_{t\to 0}\tlambda_{\ba}(t)y\in \xx_{\geq m}.
$$

\item Pour $w\in \xx_{\geq m}$, $gK\in I_{\ba}wK/K$, le réseaux $L=g\cdot \co^{d}$ admet une base unique $\{b_{i}\}_{i=1}^{d}$ sur $\co$ de la forme
$$
b_{i}=\vep^{w_{i}}\left(e_{i}+\sum_{j=1,\,j\neq i}^{d}a_{j,i}e_{j}\right),\quad a_{j,i}\in F
$$
tel que 
$$
a_{j}-a_{i}+\frac{j-i}{d}\leq \val(a_{j,i})<\alpha_{j,i}(w),\quad \text{or   } a_{j,i}=0.
$$
On a $L\in \xx_{\geq m}$ si et seulement si 
$$
\val(a_{j,i})+w_{i}\geq m,
$$ 
d'où la description précise de l'intersection $\xx_{\geq m}\cap I_{\ba}wK/K$ dans la proposition. \end{enumerate}
\end{proof}

Passant aux composantes connexes de la grassmannienne affine. Pour $v\in X_{*}(T)$, on note $\sch(v)=\overline{IvK/K}$ la variété de Schubert affine.

\begin{cor}\label{triangle1}

Soit $v\in X_{*}^{+}(T)$ tel que $v_{1}\geq v_{2}=\cdots=v_{d}$, soit $\ba=(a_{1},\cdots,a_{d})\in \bz^{d}$. Alors on a un pavage en espaces affines
$$
\sch(v)=\bigsqcup_{w\in \sch(v)^{T}}\sch(v)\cap I_{\ba}wK/K.
$$
L'intersection $\sch(v)\cap I_{\ba}wK/K=J_{\ba,v,w}wK/K$, où $J_{\ba,v,w}$ est la sous-$k$-variété ouverte et fermée de $I_{\ba}$ formée des matrices $(x_{i,j})$ 
telles que
$$
\val(x_{i,j})\geq m_{i,j},\quad \forall i\neq j,
$$ 
où $m_{i,j}=\max(a_{i}-a_{j}+\frac{i-j}{d},\,v_d-w_{j})$. De plus, l'inclusion 
$$
\sch(v)\cap I_{\ba}wK/K\subset\overline{(\sch(v)\cap I_{\ba}w'K/K)}
$$
implique que $w\prec_{I_{a}} w'$.
\end{cor}

\begin{proof}
C'est parce que la variété $\sch(v)$ est l'une des composants connexes de $\xx_{\geq v_{d}}$. 
\end{proof}

\subsubsection{Passage au dual}

On définit un accouplement $
\tr:F^{d}\times F^{d}\to F
$
par 
$$
\tr((x_{i}),\,(y_{i}))=\sum_{i=1}^{d}x_{i}y_{i}.
$$ 
Pour $L\in \xx$, on note 
$
L^{\vee}=\{y\in F^{d}\,\mid\,\tr(x,\,y)\in \co,\forall x\in L\}.
$ 
Alors $L^{\vee}$ est un réseau et $(L^{\vee})^{\vee}=L$. On a donc une involution $^{\vee}:\xx\to \xx$. Pour $m\in \bz$, on note $\xx_{\leq m}$ l'image de $\xx_{\geq -m}$ sous cette involution, alors
$$
\xx_{\leq m}=\{L\in \xx\,\mid\,L\supset \vep^{m}\co^{d}\}.
$$

\begin{lem}
Pour $g\in G(F),\, L\in \xx$, on a $(gL)^{\vee}=(g^{t})^{-1}L^{\vee}$.
\end{lem}

Utilisant ce lemme, on trouve une version duale du corollaire $\ref{triangle1}$.

\begin{cor}\label{triangle2}

Soit $v\in X_{*}^{+}(T)$ tel que $v_{1}=\cdots=v_{d-1}\geq v_{d}$, soit $\ba=(a_{1},\cdots,a_{d})\in \bz^{d}$. Alors on a un pavage en espaces affines
$$
\sch(v)=\bigsqcup_{w\in \sch(v)^{T}}\sch(v)\cap I_{\ba}wK/K.
$$
L'intersection $\sch(v)\cap I_{\ba}wK/K$ est égale à $\hat{J}_{\ba, v,w}^{-1}wK/K$, où $\hat{J}_{\ba, v,w}$ est la sous-$k$-variété ouverte et fermée de $I_{\ba}$ formée des matrices $(x_{i,j})$ telles que 
$$
\val(x_{i,j})\geq \hat{m}_{i,j},\quad \forall i\neq j,
$$ 
où $\hat{m}_{i,j}=\max(a_{i}-a_{j}+\frac{i-j}{d},\,-v_{1}+w_{i})$. De plus, l'inclusion 
$$
\sch(v)\cap I_{\ba}wK/K\subset\overline{\sch(v)\cap I_{\ba}w'K/K}
$$ 
implique que $w\prec_{I_{a}} w'$.
\end{cor}

\subsection{Pavages en tranches de la grassmannienne affine tronqu\'ee}

On fixe $v\in X_{*}^{+}(T)$ tel que $v_{1}\geq v_{2}=\cdots=v_{d-1}\geq v_{d}$. Le but de cette section est de construire une famille de pavage non standard de la variété de Schubert affine $\sch(v)$.

\subsubsection{Partition en tranches}

On va couper $\sch(v)$ en parties localement fermées. On note 
$$
R(v)=\bigcup_{i=1}^{d-1}W\cdot\{v'\in X_{*}(T)\,\mid\,v'\prec v;\,\varpi_{i}(v')=\varpi_{i}(v)\}. 
$$ 
On note $S(v)=\bigcup_{v'\in R(v)}Iv'K/K$, c'est une sous-variété ouverte de $\sch(v)$. L'idée est de couper $S(v)$ en parties localement fermées et d'utiliser le lemme suivant pour procéder par récurrence.

\begin{lem}\label{trivial1}

Soit $v\in X_{*}^{+}(T)$ tel que $v_{1}\geq v_{2}=\cdots=v_{d-1}\geq v_{d}$. Si $\sch(v)^{T}\supsetneq R(v)$, alors il existe $\bar{v}\in X_{*}^{+}(T)$ tel que 
$$
\sch(v)^{T}=\sch(\bar{v})^{T}\cup R(v),
$$
et que $\bar{v}_{1}\geq \bar{v}_{2}=\cdots=\bar{v}_{d-1}\geq \bar{v}_{d}$.
\end{lem}

\begin{proof}Dans le cas où $v_{1}>v_{2}=\cdots=v_{d}$, on a $v_{1}\geq v_{2}+d$, et $\bar{v}=(v_{1}-d+1,\,v_{2}+1,\cdots,v_{d}+1).$

Dans le cas où $v_{1}=\cdots=v_{d-1}>v_{d}$, on a $v_{1}\geq v_{d}+d$, et $\bar{v}=(v_{1}-1,\cdots,v_{d-1}-1,\,v_{d}+d-1).$

Dans le cas où $v_{1}>v_{2}=\cdots=v_{d-1}>v_{d}$, on a $\bar{v}=(v_{1}-1,\,v_{2},\cdots,v_{d-1},\,v_{d}+1).$
\end{proof}

On va ensuite couper $R(v)$ selon les sous-groupes paraboliques maximaux semi-standards. Soient $P\in \cf(T)$ maximal, $P=MN$ sa factorisation de Levi. On note $\varpi_{P}$ le poids tel que
$$
\varpi_{P}(\alpha^{\vee})=0,\;\forall \alpha\in \Phi(M,T);\quad \varpi_{P}(\alpha^{\vee})=1,\;\forall \alpha\in \Phi(N,T).
$$
On note $J_{P}\subset \{1,\cdots,d\}$ le sous-ensemble propre tel que
$$
i\in J_{P}\iff \varpi_{P}(\alpha_{i,j}^{\vee})\geq 0,\,\forall j\neq i. 
$$
On note $\bar{J}_{P}$ le complémentaire de $J_{P}$. Les $P$, $\varpi_{P}$ et $J_{P}$ se correspondent bijectivement, on les identifie en tant que sous-indice.

On prend $c\in \br$ tel que 
\begin{equation}\label{choisirc}
\begin{cases}
v_{2}<c<v_{2}+1,& \text{ si } v_{1}>v_{2}=\cdots=v_{d-1}\geq v_{d},
\\
v_{d-1}-1<c<v_{d-1},& \text{ si } v_{1}\geq v_{2}=\cdots=v_{d-1}>v_{d}.\end{cases}
\end{equation}
On note 
$$
R^{c}_{P}(v)=\{v'\in R(v)\mid \varpi_{P}(v')=\varpi_{i}(v); \,v'_{j}>c,\,\forall j\in J_{P}; \,v'_{j'}<c,\,\forall j'\notin J_{P}\},
$$
où $\varpi_{i}$ est l'unique poids fondamental dans l'orbite $W\varpi_{P}$. Il est clair que $R_{P}^{c}(v)$ ne dépend que de l'intervalle dans l'équation (\ref{choisirc}). On note $S_{P}^{c}(v)=\bigcup_{v'\in R_{P}^{c}(v)}Iv'K/K$. Alors on a la partition disjointe
$$
R(v)=\bigsqcup_{P\text{ maximal}}R_{P}^{c}(v).
$$

On va ensuite ordonner les $S_{P}^{c}(v)$. Pour $1\leq r\leq d-1$ fixé, l'union 
$
S_{r}(v)=\bigcup_{g\in W}R_{g\varpi_{r}}^{c}(v)
$
peut être ordonné par l'inverse de l'ordre de Bruhat-Tits de $g\in W$. Donc il reste à ordonner les $S_{r}(v)$, on distingue entre deux cas.


(1) $v_{1}>v_{2}=\cdots=v_{d-1}\geq v_{d}$ et $v_{2}<c<v_{2}+1$.

Pour $1\leq r\leq d-1$ tel que $v_{1}-v_{r}\geq r$, on note 
$$
v^{(r)}=(v_{1}-r+1,\,v_{2}+1,\cdots,\,v_{r}+1,\,v_{r},\cdots,\,v_{d}).
$$

\begin{lem}
\begin{enumerate}
\item
Pour $1\leq r\leq d-1$ tel que $v_{1}-v_{r}\leq r-1$, $R_{\varpi_{r}}^{c}(v)$ est vide. 

\item
Pour $1\leq r\leq d-1$ tel que $v_{1}-v_{r}\geq r$, on a
$$
R_{\varpi_{r}}^{c}(v)=\{v'\in X_{*}(T)\,\mid\,v'\prec v^{(r)};\,\varpi_{r}(v')=\varpi_{r}(v^{(r)})\}.
$$
\end{enumerate}
\end{lem}

\begin{proof}

\begin{enumerate}
\item
On raisonne par l'absurde. Supposons $w\in R_{\varpi_{r}}^{c}(v)$, alors $w_{i}>c>v_{i}$ et donc $w_{i}\geq v_{i}+1$ pour $i=2,\cdots,r$. Parce que $v_{1}\leq v_{r}+r-1<c+r-1$, on a
$$
\sum_{i=1}^{r}w_{i}> c+\sum_{i=2}^{r}(v_{i}+1)\geq c+r-1+\sum_{i=2}^{r}v_{i}\geq \sum_{i=1}^{r}v_{i},
$$
ce qui est une contradiction à l'hypothèse que $w\prec v$.

\item

C'est parce que $v^{(r)}$ est le plus longue élément dans $R_{\varpi_{r}}^{c}(v)$.
\end{enumerate}
\end{proof}

Donc, on a l'égalité
$$
\sch(\bar{v})\cup \bigcup_{i=r}^{d-1}S_{i}(v)=\sch(v^{(r)}),
$$
ce qui donne l'ordre entre les $S_{i}(v)$: on a
$
S_{d-1}(v)\prec S_{d-2}(v)\prec\cdots \prec S_{1}(v).
$
La figure \ref{order1} donne un exemple de l'ordre de pavage pour $\gl_{3}$ dans ce cas.

\begin{figure}
\begin{center}

\begin{tikzpicture}
\draw (-3.5, -2)--(3.5,2);
\draw (-3.5, 2)--(3.5,-2);
\draw (0, 3)--(0,-3.5);

\draw[red, dashed] (-2.5, 2.7) node[anchor=east]{$v_{1}=c$}--(1,-3.3);
\draw[red, dashed] (2.5, 2.7) node[anchor=east]{$v_{2}=c$}--(-1,-3.3);
\draw[red, dashed] (-3.7, 0.8) node[anchor=east]{$v_{3}=c$}--(3.7,0.8);

\filldraw[black] (-1.7, 1.6) circle (1pt)
                 (1.7,1.6) circle (1pt)                 
                 (-2.2, 0.67) circle (1pt)
                 (2.2, 0.67) circle (1pt)
                 (0.55, -2.3) circle (1pt)                 
                 (-0.55, -2.3) circle (1pt)
                 (2, 1.6) circle (1pt)
                 (2.35, 0.93) circle (1pt)
                 (-2, 1.6) circle (1pt)
                 (-2.35, 0.93) circle (1pt)                 
                 (-0.4, -2.56) circle (1pt)
                 (0.4, -2.56) circle (1pt);
                               
\draw (-1.7, 1.6)-- node[anchor=south]{\color{blue} 3} (1.7,1.6);
\draw (-2.2, 0.67)-- node[anchor=east]{\color{blue}1}(-0.55,-2.3);
\draw (2.2, 0.67)-- node[anchor=west]{\color{blue}2}(0.55,-2.3);

\draw (2, 1.6)--node[anchor=west]{\color{blue}6} (2.35, 0.93);
\draw (-2, 1.6)--node[anchor=east]{\color{blue}5} (-2.35,0.93);
\draw (-0.4, -2.56)-- node[anchor=north]{\color{blue}4}(0.4,-2.56);

\draw[->, thick](0,0)--(0.5,0) node[anchor=west]{$\alpha$};
\draw[->, thick](0,0)--(-0.25, 0.433) node[anchor=south]{$\beta$};
\end{tikzpicture}
\caption{L'ordre de pavage pour $\gl_{3}$--premier cas.}
\label{order1}
\end{center}
\end{figure}
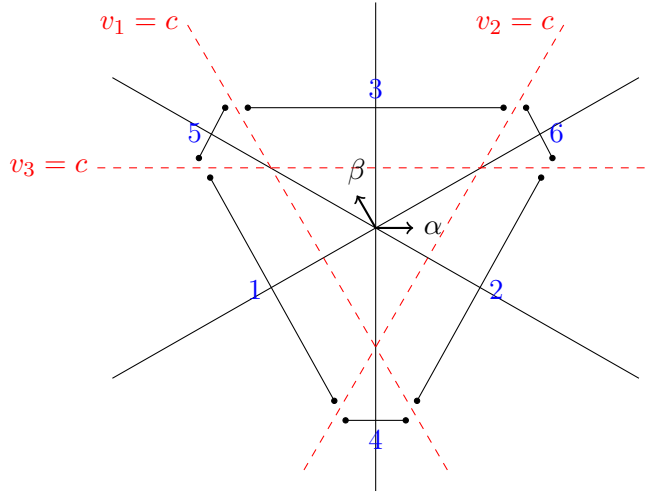

(2) $v_{1}\geq v_{2}=\cdots=v_{d-1}>v_{d}$ et $v_{d-1}-1<c<v_{d-1}$.

Pour $1\leq r\leq d-1$ tel que $v_{r+1}-v_{d}\geq d-r$, on note
$$
v_{(r)}=(v_{1},\,\cdots,\,v_{r},\,v_{r+1}-1,\cdots,v_{d-1}-1,\,v_{d}+d-r-1). 
$$
Alors parallèlement on a

\begin{lem}
\begin{enumerate}
\item
Pour $1\leq r\leq d-1$ tel que $v_{r+1}-v_{d}<d-r$, $R_{\varpi_{r}}^{c}(v)$ est vide. 
\item
Pour $1\leq r\leq d-1$ tel que $v_{r+1}-v_{d}\geq d-r$, on a 
$$
R_{\varpi_{r}}^{c}(v)=\{v'\in X_{*}(T)\,\mid\,v'\prec v_{(r)};\,\varpi_{r}(v')=\varpi_{r}(v_{(r)})\}.
$$
\end{enumerate} 
\end{lem}

Donc, on a l'égalité
$$
\sch(\bar{v})\cup \bigcup_{i=1}^{r}S_{i}(v)=\sch(v_{(r)}),
$$
ce qui donne l'ordre entre les $S_{i}(v)$: on a $S_{1}(v)\prec S_{2}(v)\prec\cdots\prec S_{d-1}(v)$. La figure \ref{order2} donne un exemple de l'ordre de pavage pour $\gl_{3}$ dans ce cas.

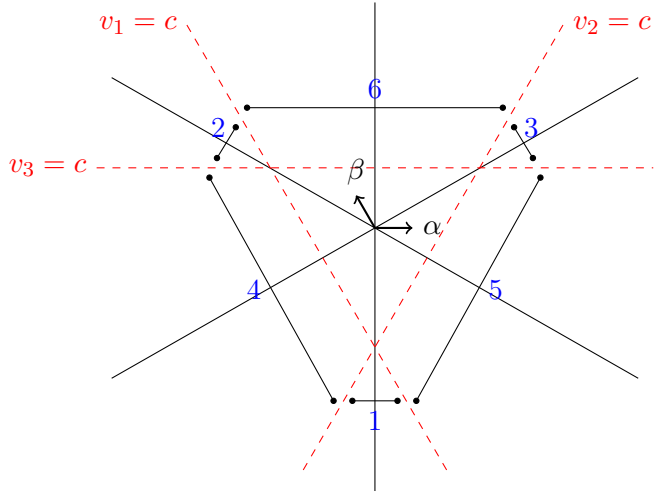
\begin{figure}[htbp]
\begin{center}

\begin{tikzpicture}
\draw (-3.5, -2)--(3.5,2);
\draw (-3.5, 2)--(3.5,-2);
\draw (0, 3)--(0,-3.5);

\draw[red, dashed] (-2.5, 2.7) node[anchor=east]{$v_{1}=c$}--(1,-3.3);
\draw[red, dashed] (2.5, 2.7) node[anchor=west]{$v_{2}=c$}--(-1,-3.3);
\draw[red, dashed] (-3.7, 0.8) node[anchor=east]{$v_{3}=c$}--(3.7,0.8);

\filldraw[black] (-1.7, 1.6) circle (1pt)
                 (1.7,1.6) circle (1pt)                 
                 (-2.2, 0.67) circle (1pt)
                 (2.2, 0.67) circle (1pt)
                 (0.55, -2.3) circle (1pt)                 
                 (-0.55, -2.3) circle (1pt)
                 (1.85, 1.34) circle (1pt)
                 (2.1, 0.93) circle (1pt)
                 (-1.85, 1.34) circle (1pt)
                 (-2.1, 0.93) circle (1pt)                 
                 (-0.3, -2.3) circle (1pt)
                 (0.3, -2.3) circle (1pt);
                                  
\draw (-1.7, 1.6)-- node[anchor=south]{\color{blue}6} (1.7,1.6);
\draw (-2.2, 0.67)-- node[anchor=east]{\color{blue}4}(-0.55,-2.3);
\draw (2.2, 0.67)-- node[anchor=west]{\color{blue}5}(0.55,-2.3);

\draw (1.85, 1.34)node[anchor=west]{\color{blue}3} --(2.1,0.93);
\draw (-1.85, 1.34)node[anchor=east]{\color{blue}2}-- (-2.1,0.93);
\draw (-0.3, -2.3)-- node[anchor=north]{\color{blue}1}(0.3,-2.3);

\draw[->, thick](0,0)--(0.5,0) node[anchor=west]{$\alpha$};
\draw[->, thick](0,0)--(-0.25, 0.433) node[anchor=south]{$\beta$};

\end{tikzpicture}

\caption{L'ordre de pavage pour $\gl_{3}$--deuxième cas.}
\label{order2}
\end{center}
\end{figure}

\subsubsection{Pavage non standard en tranches}

On va repaver les $S_{P}^{c}(v)$ en espaces affines. Pour cela, on montre qu'ils sont des fibrations en espaces affines sur certaines variétés de Schubert affines de $\xx^{M}$, et on se ramène à repaver ces variétés de Schubert affines.

\begin{lemdef}\label{quotvect}

Soit $Z$ un sous-$k$-schéma réduit de $\xx$ de type fini. Soit $V$ un sous-$\co$-module de type fini de $\kg_{F}$. On suppose que $\dim_{k}(V/V\cap \Ad(g)\kg(\co))$ est indépendent de $g$ pour $gK\in Z$, alors les $V/V\cap \Ad(g)\kg(\co)$ s'organisent en un fibré vectoriel que on notera $\tilde{V}/\tilde{V}\cap \ckk$ sur $Z$.
\end{lemdef}

\begin{proof}

Choisissons $N\in \bn$ assez grand tel que
$$
\vep^{N}\kg(\co)\subset \Ad(g)\kg(\co)\subset \vep^{-N}\kg(\co),\;\forall gK\in Z.
$$
On note $L_{N}$ le fibré vectoriel constant sur $Z$ avec fibre $\vep^{-N}\kg(\co)/\vep^{N}\kg(\co)$. On note $\ckk'$ le sous-fibré vectoriel de $L_{N}$ tel que sa fibre sur $gK$ est $\Ad(g)\kg(\co)/\vep^{N}\kg(\co)$. On note $\cvv'$ le sous-fibré vectoriel de $L_{N}$ tel que sa fibre sur $gK$ est $(V+\vep^{N}\kg(\co))/\vep^{N}\kg(\co)$. L'hypothèse implique que l'intersection $\cvv'\cap \ckk'$ est équi-dimensionel sur $Z$, et donc il est un fibré vectoriel sur $Z$. Le quotient $\cvv'/\cvv'\cap \ckk'$ est donc le fibré vectoriel que l'on cherche.
\end{proof}

Pour $x=gK\in \xx$, soit $g=nmk,\,n\in N(F),m\in M(F),\,k\in K$ la décomposition d'Iwasawa de $g$, on note $x_{P}$ le point $mK^{M}\in \xx^{M}$, qui ne dépend que de $x$ et de $P$.

\begin{defn}

On note $f_{P}:\xx\to \xx^{M}$ la rétraction $f_{P}(x)=x_{P},\,\forall x\in \xx$.

\end{defn}

\begin{lem}\label{affine1}

La rétraction
$$
f_{P}:S_{P}^{c}(v)\to S_{P}^{c}(v)\cap \xx^{M}
$$
est une fibration en espaces affines, en particulier sa restriction à $IwK/K$ l'est aussi pout tout $ w\in R_{P}^{c}(v)$. 
\end{lem}

\begin{proof}

Par définition, on a
$$
S_{P}^{c}(v)\cap \xx^{M}=\bigcup_{w\in R_{P}^{c}(v)}I^{M}wK/K.
$$

On note $N_{I}=N(F)\cap I$. Pour $mwK\in S_{P}^{c}(v)\cap \xx^{M},\, m\in I^{M}$, on a
$$
\frac{\kn_{I}}{\kn_{I}\cap \Ad(mw)\kg(\co)}\cong \frac{\kn_{I}}{\kn_{I}\cap \Ad(w)\kg(\co)},
$$
car $I^{M}$ normalise $\kn_{I}$. Parce que $w\in R_{P}^{c}(v)$, la dimension de la dernier terme est
\begin{eqnarray*}
\sum_{\alpha\in \Phi(N,T)}(\alpha(w)+\lfloor\alpha(x_{0})\rfloor)&=&d\varpi_{P}(w)+\sum_{\alpha\in \Phi(N,T)}(\lfloor\alpha(x_{0})\rfloor)\\
&=&d\varpi_{i}(v)+\sum_{\alpha\in \Phi(N,T)}(\lfloor\alpha(x_{0})\rfloor)
\end{eqnarray*}
où $x_{0}=(1,(d-1)/d,\cdots,2/d,1/d)\in \kt$ et $\varpi_{i}$ est l'unique poids fondamental dans l'orbite $W\varpi_{P}$, donc la dimension est constante pour $v\in R_{P}^{c}(v)$. Donc les dimensions de $\kn_{I}/\kn_{I}\cap \Ad(g)\kg(\co)$ sont constantes pour $gK\in S_{P}^{c}(v)\cap \xx^{M}$. D'après le lemme \ref{quotvect}, ils s'organisent en un fibré vectoriel $\tilde{\kn}_{I}/\tilde{\kn}_{I}\cap \ckk$.

Par l'isomorphisme canonique
$$
f_{P}^{-1}(mwK)\cong \frac{\kn_{I}}{\kn_{I}\cap \Ad(mw)\kg(\co)},
$$
on obtient que $S_{P}^{c}(v)$ est un $\tilde{\kn}_{I}/\tilde{\kn}_{I}\cap \ckk$-torseur sur $S_{P}^{c}(v)\cap \xx^{M}$, d'où la proposition.

\end{proof}

Pour $\ba=(a_{1},\cdots,a_{d})\in \bz^{d}$, on dit que $\ba$ est \emph{positif} par rapport à $P$ si 
$$
a_{i}-a_{j}>0,\quad\forall i\in J_{P},j\in \bar{J}_{P}.
$$ 
On dit qu'il est \emph{négatif} par rapport à $P$ si $-\ba$ est positif par rapport à $P$.

\begin{prop}\label{tranche}

Soit $v\in X_{*}^{+}(T)$ tel que $v_{1}\geq v_{2}=\cdots=v_{d-1}\geq v_{d}$. On prend un nombre $c\in \br\backslash \bz$ comme dans l'équation (\ref{choisirc}). Soient $P\in \cf(T)$ maximal, $P=MN$ sa factorisation de Levi. Soit $\ba\in \bz^{d}$ tel que $\ba$ est négatif par rapport à $P$.  Alors on a le pavage en espaces affines
$$
S_{P}^{c}(v)=\bigsqcup_{w\in R_{P}^{c}(v)}S_{P}^{c}(v)\cap I_{\ba}wK/K,
$$
l'inclusion $S_{P}^{c}(v)\cap I_{\ba}wK/K\subset \overline{S_{P}^{c}(v)\cap I_{\ba}w'K/K}$ implique que $w\prec_{I_{\ba}}w'$. De plus, on a l'intersection 
$$
S_{P}^{c}(v)\cap I_{\ba}wK/K=\big(N(F)\cap I\big)H wK/K,
$$ 
où $H=H_{1}\times H_{2}^{-1}$ avec

\begin{enumerate}

\item    $H_{1}$ est la sous-$k$-variété ouverte et fermée de $\gl_{J_{P}}(F)$ formée des matrices $(x_{i,j})$ telles que $x_{i,i}\in \co$ et que
\begin{eqnarray*}
\val(x_{i,j})\geq m_{i,j},\quad \forall i,j\in J_{P};\,i\neq j,\end{eqnarray*}
où $m_{i,j}=\max(a_{i}-a_{j}+\frac{i-j}{d},\,\lceil c\rceil-w_{j})$.

\item   $H_{2}$ est la sous-$k$-variété ouverte et fermée de $\gl_{\bar{J}_{P}}(F)$ formée des matrices $(x_{i,j})$ telles que $x_{i,i}\in \co$ et que
\begin{eqnarray*}
\val(x_{i,j})\geq \hat{m}_{i,j}, \quad \forall i,j\in \bar{J}_{P}; i\neq j,
\end{eqnarray*}
où $\hat{m}_{i,j}=\max(a_{i}-a_{j}+\frac{i-j}{d},\,-\lfloor c \rfloor+w_{i})$. 
\end{enumerate}

\end{prop}

\begin{proof}

Soit $\tlambda_{\ba}\in X_{*}(\mt)$ le co-caractère défini par l'équation (\ref{tlam}). Pour $w\in X_{*}(T)$, on a
$$
I_{\ba}wK/K=\{x\in \xx\,\mid\,\lim_{t\to 0}\tlambda_{\ba}(t)x=wK\}.
$$
Donc 
$$
S_{P}^{c}(v)\cap I_{\ba}wK/K=\{x\in S_{P}^{c}(v)\,\mid\,\lim_{t\to 0}\tlambda_{\ba}(t)x=wK\}.
$$
Pour $u\in N(F)\cap I,\,x\in S_{P}^{c}(v)$, on a 
$$
\lim_{t\to 0}\tlambda_{\ba}(t)ux=\lim_{t\to 0}[\Ad(\tlambda_{\ba}(t))u]\tlambda_{\ba}(t)x=\lim_{t\to 0}\tlambda_{\ba}(t)x,
$$
car $\ba$ est négatif par rapport à $P$, d'où l'égalité
\begin{equation}\label{limitiwahori}
\lim_{t\to 0}\tlambda_{\ba}(t)x=\lim_{t\to 0}\tlambda_{\ba}(t)[f_{P}(x)].
\end{equation}
Cette égalité implique que le pavage non standard se factorise par la fibration $f_{P}:S_{P}^{c}(v)\to S_{P}^{c}(v)\cap \xx^{M}$. Plus précisément,
\begin{equation}\label{reducible}
S_{P}^{c}(v)\cap I_{\ba}wK/K=(N(F)\cap I)\cdot [(S_{P}^{c}(v)\cap \xx^{M})\cap I^{M}_{\ba}wK/K].
\end{equation}
Remarquons que 
\begin{eqnarray*}
S_{P}^{c}(v)\cap \xx^{M}&=&\bigcup_{w\in R_{P}^{c}(v)}I^{M}wK^{M}/K^{M}\\
&=&\xx^{\gl_{J_{P}},(n_{1})}_{\geq \lceil c\rceil}\times \xx^{\gl_{\bar{J}_{P}},(n_{2})}_{\leq \lfloor c\rfloor},
\end{eqnarray*}
pour certaines indices $n_{1},n_{2}\in \bz$.

Donc la proposition se ramène aux corollaires \ref{triangle1} et \ref{triangle2}, pour $\xx^{\gl_{J_{P}},(n_{1})}_{\geq \lceil c\rceil}$ et $\xx^{\gl_{\bar{J}_{P}},(n_{2})}_{\leq \lfloor c\rfloor}$ respectivement.
\end{proof}

\begin{cor}\label{affine2}

Même hypothèse que la proposition précédente. La rétraction 
$$
f_{P}: S_{P}^{c}(v)\cap I_{\ba}wK/K \to S_{P}^{c}(v)\cap I^{M}_{\ba}wK^{M}/K^{M}
$$
est une fibration en espaces affines.

\end{cor}

\begin{proof}
C'est parce que le pavage non standard factorise par la fibration $f_{P}:S_{P}^{c}(v)\to S_{P}^{c}(v)\cap \xx^{M}$, qui est une fibration en espaces affines par le lemme \ref{affine1}. 

\end{proof}

En conclusion, soient $v\in X_{*}^{+}(T)$ tel que $v_{1}\geq v_{2}=\cdots=v_{d-1}\geq v_{d}$, $c\in \br\backslash \bz$ un nombre comme dans l'équation (\ref{choisirc}). \'Etant donné $\ba_{P}\in \bz^{d}$ pour tout $P\in \cf(T)$ maximal tel que $\ba_{P}$ est négatif par rapport à $P$, on peut paver $S_{P}^{c}(v)$ en espaces affines avec l'Iwahori $I_{\ba_{P}}$ d'après la proposition \ref{tranche}. De cette manière, on construit un pavage non standard de 
$S(v)$. Comme on a remarqué, cette processus peut être continué sur $\sch(\bar{v})$ avec autre paramètre $c\in \br$. Par récurrence, on construit un pavage non standard de $\sch(v)$, on l'appelle \emph{le pavage en tranches} de $\sch(v)$.

\section{Application aux pavages de la fibre de Springer affine}

Soit $\gamma\in \kg(F)$ un élément semi-simple régulier entier, on va utiliser les pavages non standard pour paver la fibre de Springer affine $\xx_{\gamma}$. La fibre de Springer affine n'est pas réduite comme un schéma, mais on va travailler avec sa structure réduite puisque on s'intéresse qu'à sa cohomologie étale.

\subsection{Une proposition technique}

Prenons un sous-groupe parabolique maximal $P=MN\in \cf(T)$. Soit $\gamma\in \km(F)\subset \kg(F)$ un élément semi-simple régulier entier. Soient $H\subset M(F),\,U\subset N(F)$ des sous-groupes ouverts et fermés.

\begin{lem}\label{tech0}

Considérons la rétraction
$$
f_{P}:U\xx_{\gamma}^{M}\cap \xx_{\gamma}\to \xx^{M}_{\gamma}.
$$
Soit $gK^{M}\in \xx^{M}_{\gamma}$, alors $f_{P}^{-1}(gK^{M})$ est isomorphe canoniquement à
$$
\ker\left\{\ad(\gamma):\,\frac{\ku}{\ku\cap \Ad(g)\kg(\co)}\to \frac{\ad(\gamma)\ku}{\ad(\gamma)\ku \cap \Ad(g)\kg(\co)}\right\},
$$ 
\end{lem}

\begin{proof}

Pour $gK^{M}\in \xx_{\gamma}^{M}$, soit $u\in \ku$, alors 
\begin{eqnarray*}
(1+u)gK\in \xx_{\gamma}&\iff&\Ad(1+u)^{-1}\gamma\in \Ad(g)\kg(\co)\\
&\iff& \gamma-[u,\,\gamma]\in \Ad(g)\kg(\co)\\
&\iff& [u,\,\gamma]\in \Ad(g)\kg(\co),
\end{eqnarray*}
d'ou le lemme (dans la deuxième ligne on utilise le fait que $P$ est maximal). \end{proof}

\begin{prop}\label{techfam}

Soit $X$ une sous-variété de $\xx^{M}_{\gamma}$. On suppose que les dimensions 
$$
\dim_{k}\left(\frac{\ku}{\ku\cap \Ad(g)\kg(\co)}\right),\quad \dim_{k}\left(\frac{\ad(\gamma)\ku}{\ad(\gamma)\ku \cap \Ad(g)\kg(\co)}\right),
$$ 
sont indépendantes de $g$ pour tout $gK\in X$. Alors la rétraction
$$
f_{P}:UX\cap \xx_{\gamma}\to X
$$
est une fibration en espaces affines. 
\end{prop}

\begin{proof}

D'après le lemme \ref{quotvect}, les $\ku/\ku\cap \Ad(g)\kg(\co),\, gK\in X$, s'organisent en un fibré vectoriel $\tilde{\ku}/\tilde{\ku}\cap \ckk$. Il en est de même pour $\ku':=\ad(\gamma)\ku$. Donc le noyau
$$\ckk_{\ku,\gamma}:=\ker\left\{\ad(\gamma):\,\frac{\tilde{\ku}}{\tilde{\ku}\cap \ckk}\to \frac{\tilde{\ku}'}{\tilde{\ku}'\cap \ckk}\right\}
$$
est un fibré vectoriel sur $X$ car le morphisme est surjectif. D'après le lemme \ref{tech0}, la fibration en question est un $\ckk_{\ku, \gamma}$-torseur, donc elle est une fibration en espaces affines.
\end{proof}

\begin{rem}
La condition que la dimension de $\ku/\ku\cap \Ad(g)\kg(\co)$ est indépendante de $g$ pour tout $gK\in X$ est équivalent à la condition que la rétraction
$$
f_{P}:UX\to X
$$
est une fibration en espaces affines. Dans la suite, on utilise aussi cette condition alternative.
\end{rem}

\begin{lem}\label{normal}

Supposons que $H$ normalise $U$. Soit $X$ une $H$-orbite dans $\xx^{M}$, alors la dimension $\ku/\ku\cap \Ad(g)\kg(\co)$ est indépendante de $g$ pour $gK\in X$.
\end{lem}

\begin{proof}

Fixe $gK^{M}\in X$, pour tout $h\in H$, on a 
$$
\frac{\ku}{\ku\cap \Ad(hg)\kg(\co)}\cong \frac{\ku}{\ku\cap \Ad(g)\kg(\co)},$$
car $H$ normalise $U$, d'où le lemme.
\end{proof}

\begin{cor}\label{technormal}

Supposons que $H$ normalise les algèbres de Lie $\ku$ et $\ad(\gamma)\ku$. Soit $X$ une $H$-orbite dans $\xx^{M}$. Alors la rétraction 
$$
f_{P}:UX\cap \xx_{\gamma}\to X\cap \xx^{M}_{\gamma}
$$
est une fibration en espaces affines.\end{cor}

\subsection{Pavage pour $\gl_3$}

Dans cette section uniquement, on suppose que  $\mathrm{char}(k)>3$. Alors le tore maximal $Z_{G}(\gamma)$ est isomorphe soit à $F^{\times}\times F^{\times}\times F^{\times}$, soit à $F^{\times}\times F((\vep^{1/2}))^{\times}$, soit à $F((\vep^{1/3}))^{\times}$. On appelle $\gamma$ dans ces cas \emph{non-ramifié, mélangé, elliptique} respectivement.

Le pavage de $\xx_{\gamma}$ est connue dans les cas suivants:

\begin{thm}[Lucarelli]

Pour $\gamma$ non-ramifié, $\xx_{\gamma}$ admet un pavage en espaces affines.
\end{thm}

Pour $\gamma$ elliptique, le pavage de $\xx_{\gamma}$ est donné par Goresky, Kottwitz et Macpherson dans \cite{gkm1} car $\gamma$ est forcément équivalué.

\begin{prop}

Pour $\gamma$ mélangé, la cohomologie de $\xx_{\gamma}$ est pure.
\end{prop}

\begin{proof}

Pour $\gamma$ équivalué, c'est déjà montré par Goresky, Kottwitz et Macpherson. Pour $\gamma$ non-equivalué, à conjugaison près, on peux supposer que
$$
\gamma=\begin{bmatrix}a\vep^{n_{1}}&&\\ &&\vep^{n_{2}}\\&b\vep^{n_{2}+1}\end{bmatrix},\quad n_{1}\leq n_{2},\,a,b\in \co^{\times}.
$$

On va paver $\xx_{\geq -m}\cap \xx_{\gamma}$ en espaces affines. D'abord, on pave $\xx_{\geq -m}$ avec l'Iwahori $I'=\Ad(\diag(\vep^{3m},1,1))I$. Pour $w\in \xx_{\geq -m}^{T}$, on note $C(w)=\xx_{\geq -m}\cap I'wK/K$. Par la proposition \ref{pav1}, on a un pavage en espaces affines $
\xx_{\geq -m}=\bigsqcup_{w\in \xx_{\geq -m}^{T}}C(w)$ et 
$$
C(w)= \begin{bmatrix}\co&&\\ \kp^{m_{w}}&\co&\co\\ \kp^{m_{w}}&\kp&\co\end{bmatrix}wK/K,\quad m_{w}=-m-w_{1}.
$$

On note 
$$
P=\begin{bmatrix}*&&\\ *&*&*\\  *&*&*\end{bmatrix}.$$ 
Soit $P=MN$ sa factorisation de Levi standard. D'après le corollaire \ref{technormal}, la rétraction 
\begin{equation*}\label{eqnd1}
f_{P}:C(w)\cap \xx_{\gamma}\to I^{M}wK^{M}/K^{M}\cap \xx_{\gamma}^{M}
\end{equation*}
est une fibration en espaces affines. En effet, on prend 
$$
U=\begin{bmatrix}1&&\\ \kp^{m_{w}}&1&\\  \kp^{m_{w}}&&1\end{bmatrix},
$$ 
alors $C(w)=UI^{M}wK/K$. Il est évident que $I^{M}$ normalise $\ku=\lie(U)$ et $\ad(\gamma)\ku$.

De plus, l'intersection $I^{M}wK^{M}/K^{M}\cap \xx^{M}_{\gamma}$ est isomorphe à un espace affine car $M=\gl_{1}\times \gl_{2}$, donc $C(w)\cap \xx_{\gamma}$ l'est aussi.

\end{proof}

Avec le même pavage, on peut montrer:

\begin{prop}\label{d1}

Soit $\gamma=(\gamma_{1},\gamma_{2})\in (\ggl_{1}\times\ggl_{d-1})(F)\subset\ggl_{d}(F)$ tel que $\alpha_{1,i}(\gamma)=n_{1},\,i=2,\cdots,d$ et $\gamma_{2}\in \ggl_{d-1}(F)$ est équivalué de valuation $n_{2}+r,\,r\in \bq,\,0\leq r<1$, $n_{1}\leq n_{2},\,n_{1},n_{2}\in \bn$. Alors $\xx_{\gamma}$ est pur.\end{prop}

En conclusion, on a

\begin{thm}\label{gl3}
Pour tout $\gamma\in \ggl_{3}(F)$ semi-simple régulier entier, la cohomologie de $\xx_{\gamma}$ est pure.\end{thm}

Dans la suit, on va utiliser les pavages non standard en tranches de la grassmannienne affine tronquée que l'on a construit dans \S2.4, pour paver les fibres de Springer affines $\xx_{\gamma}$, $\gamma$ non-ramifié. Ce pavage est plus ``flexible'' que celui de Lucarelli, et il est indisponible pour paver la fibre de Springer affine pour $\gl_{4}$ dans le cas non-ramifié.

\begin{thm}\label{famgl3}

Soient $\gamma\in \kt(\co)$ un élément régulier, $v\in X_{*}^{+}(T)$, $c\in \br\backslash \bz$ un nombre comme dans l'équation (\ref{choisirc}). Alors pour tout $P\in \cf(T)$ maximal, l'intersection $S_{P}^{c}(v)\cap \xx_{\gamma}$ admet un pavage en espaces affines.

\end{thm}

\begin{proof}

Après conjuguer par le groupe de Weyl, on peut supposer que 
$$
\val(\alpha_{1,2}(\gamma))=\val(\alpha_{1,3}(\gamma))=n_{1},\quad \val(\alpha_{2,3}(\gamma))=n_{2}, \quad n_{1}\leq n_{2}.
$$
On note $\ba=(n_{1},n_{2},n_{2})$. Soit $P=MN$ la factorisation de Levi standard de $P$. On note $U=N(F)\cap I$, alors on a $IwK/K=UI^{M}wK/K$. 

\begin{enumerate}

\item Pour $\varpi_{P}=\varpi_{1}$ ou $-\varpi_{1}$, on va montrer que l'intersection $IwK/K\cap \xx_{\gamma}$ est isomorphe à un espace affine standard. On traite que $\varpi_{P}=\varpi_{1}$, l'autre cas étant pareil.

Parce que $I^{M}$ normalise $\ku=\lie(U)$ et $\ad(\gamma)\ku$,  d'après le corollaire \ref{technormal}, la rétraction 
\begin{equation*}
f_{P}:IwK/K\cap \xx_{\gamma}\to I^{M}wK^{M}/K^{M}\cap \xx_{\gamma}^{M}
\end{equation*} 
est une fibration en espaces affines. De plus, l'intersection $I^{M}wK^{M}/K^{M}\cap\xx^{M}_{\gamma}$ est isomorphe à un espace affine car $M=\gl_{1}\times \gl_{2}$. L'intersection $IwK/K\cap \xx_{\gamma}$ est donc isomorphe à un espace affine.

\item Pour $\varpi_{P}=\varpi_{2}$ ou $s_{23}\varpi_{2}$. On ne traite que $\varpi_{P}=\varpi_{2}$, l'autre cas étant pareil.

Soit $m_{w}=\lceil c \rceil-w_{2}$, et 
$$
H=I^{M}_{\ba}\cap \Ad(\diag(\vep^{m_{w}},1,1))K^{M},
$$
qui est un groupe de Lie. D'après la proposition \ref{tranche}, on a un pavage en espaces affines
$$
S_{P}^{c}(v)=\bigsqcup_{w\in R_{P}^{c}(v)}S_{P}^{c}(v)\cap I_{\ba}wK/K,
$$ 
et
$$
S_{P}^{c}(v)\cap I_{\ba}wK/K=UHwK/K.
$$
On va montrer que l'intersection $S_{P}^{c}(v)\cap I_{\ba}wK/K\cap \xx_{\gamma}$ est isomorphe à un espace affine standard. Ainsi on obtiendra un pavage en espaces affines
$$
S_{P}^{c}(v)\cap \xx_{\gamma}=\bigsqcup_{w\in R_{P}^{c}(v)}S_{P}^{c}(v)\cap I_{\ba}wK/K\cap \xx_{\gamma}.
$$
Par le corollaire \ref{affine2}, la rétraction 
$$
f_{P}: UHwK/K\to HwK^{M}/K^{M}
$$ 
est une fibration en espaces affines. Puisque $H\subset I_{\ba}^{M}$, il normalise $\ad(\gamma)\ku$, le lemme $\ref{normal}$ et la proposition \ref{techfam} impliquent que la rétraction 
\begin{equation*}
f_{P}:UHwK/K\cap \xx_{\gamma}\to HwK^{M}/K^{M}\cap \xx_{\gamma}^{M}
\end{equation*} 
est une fibration en espaces affines. Puisque $M=\gl_{2}\times \gl_{1}$, l'intersection $HwK^{M}/K^{M}\cap \xx_{\gamma}^{M}$ est isomorphe à un espace affine. Par conséquent, $S_{P}^{c}(v)\cap I_{\ba}wK/K\cap \xx_{\gamma}$ est isomorphe à un espace affine.

\item Pour $\varpi_{P}=-\varpi_{2}$ ou $-s_{23}\varpi_{2}$. On va montrer que l'intersection $S_{P}^{c}(v)\cap I_{-\ba}wK/K\cap \xx_{\gamma}$ est isomorphe à un espace affine standard, ainsi on obtiendra un pavage en espaces affines de $S_{P}^{c}(v)\cap \xx_{\gamma}$ similaire au cas précedent. On ne traitre que le cas $\varpi_{P}=-\varpi_{2}$, l'autre cas étant pareil.

Soit $m_{w}=-\lfloor c \rfloor+w_{2}$, on note 
$$
H=I^{M}_{-\ba}\cap \Ad(\diag(1,\vep^{m_{w}},\vep^{m_{w}}))K^{M}.
$$
qui est un groupe de Lie. D'après la proposition \ref{tranche}, on a 
$$
S_{P}^{c}(v)\cap I_{-\ba}wK/K=UHwK/K.
$$
Par le corollaire \ref{affine2}, la rétraction 
$$
f_{P}: UHwK/K\to HwK^{M}/K^{M}
$$ 
est une fibration en espaces affines. Puisque $H\subset I^{M}_{-\ba}$, il normalise $\ad(\gamma)\ku$. Le lemme $\ref{normal}$ et la  proposition \ref{techfam} impliquent que la rétraction 
\begin{equation*}
f_{P}:UHwK/K\cap \xx_{\gamma}\to HwK^{M}/K^{M}\cap \xx_{\gamma}^{M}
\end{equation*} 
est une fibration en espaces affines. Puisque $M=\gl_{2}\times \gl_{1}$, l'intersection $HwK^{M}/K^{M}\cap \xx_{\gamma}^{M}$ est isomorphe à un espace affine et donc $S_{P}^{c}(v)\cap I_{-\ba}wK/K\cap \xx_{\gamma}$ est isomorphe à un espace affine.

\end{enumerate}
\end{proof}

En utilisant le pavage en tranches de $\sch(v)$ que l'on a construit dans \S 2.4, on obtient

\begin{cor}\label{gl3alt}

Pour tout $v\in X_{*}^{+}(T)$, la sous-variété fermée $\sch(v)\cap \xx_{\gamma}$ de $\xx_{\gamma}$ admet un pavage en espaces affines. En particulier, elle est pure.
\end{cor}

\subsection{Pavage pour $\gl_4$ dans le cas nonramifi\'e}

Soit $\gamma\in \kt(\co)$ un élément régulier entier, le but de cette section est de démontrer le théorème suivant:

\begin{thm}\label{principal}

Pour tout $m\in \bn$, l'intersection $\xx_{\geq -m}\cap \xx_{\gamma}$ admet un pavage en espaces affines. Par conséquent, la cohomologie de $\xx_{\gamma}$ est pure.
\end{thm}

Le reste de la section est consacré à la démonstration du théorème. Quitte à conjuguer par le groupe de Weyl, on peut supposer que $\gamma\in \kt(\co)$ est en forme minimale (voir l'appendice), alors sa valuation radicielle est l'un des deux types suivants:

\begin{enumerate}

\item $(n_{1}, n_{2}, n_{3}), \, n_{1}\leq n_{2}\leq n_{3}$, 

\item $(n_{1},\, n,\,n_{2}), \,n\leq n_{1}\leq n_{2}.$
\end{enumerate}

Pour les deux types, il suffit de paver l'intersection de $\xx_{\gamma}$ avec le composant neutre $\xx_{\geq -m}^{(0)}$. Pour simplifier les notations, on le note encore $\xx_{\geq -m}$. On commence par paver $\xx_{\geq -m}$ en utilisant le sous-groupe d'Iwahori 
$$
I'=\Ad(\diag(\vep^{4m},1,1,1))I.
$$ 
On note $C(w)=\xx_{\geq -m}\cap I'wK/K$. D'après la proposition $\ref{pav1}$, on a un pavage en espaces affines 
$$
\xx_{\geq -m}=\bigsqcup_{w\in \xx_{\geq -m}^{T}}C(w)
$$ 
et 
$$
C(w)= \begin{bmatrix}\co&&&\\ \kp^{m_{w}}&\co&\cdots&\co\\ \vdots&\vdots&\ddots&\vdots\\ \kp^{m_{w}}&\kp&\cdots&\co\end{bmatrix}wK/K,\quad m_{w}=-m-w_{1}.
$$
Pour $-m\leq b \leq 3m,\, b\in \bz$, on note 
$$
R_{b}=\{w\in \xx_{\geq -m}^{T}\mid w_{1}=b\},\quad V_{b}=\bigsqcup_{w\in R_{b}}C(w).
$$
Alors $\xx_{\geq -m}=\bigcup_{b=-m}^{3m}V_{b}$, et pour tout $b$, on a 
$
\overline{V}_{b}=\bigcup_{i=b}^{3m}V_{i},
$
ce qui donne l'ordre de pavage entre les $V_{b}$. Voir la figure \ref{gl4vvv} pour ce découpage. Pour démontrer le théorème \ref{principal}, il suffit donc de paver chaque $V_{b}\cap \xx_{\gamma}$.

\begin{figure}[htbp]
\begin{center}

\begin{tikzpicture}

\draw (-2,1)--(2,1);
\draw (-2,1)--(0,-2.2);
\draw (2,1)--(0,-2.2);
\draw (-0.5,1.8)--(-2,1);
\draw (-0.5,1.8)--(2,1);
\draw[dashed] (-0.5,1.8)--(0,-2.2);
\draw [red](0.125,1.6) node[anchor=south]{$R_{b}$} --(-1,1);
\draw[red] (-1,1)--(0.5,-1.4);
\draw[dashed, red] (0.125,1.6)--(0.5,-1.4);

\draw[->, thick] (0,0.2)--(0.5,0.2) node[anchor=west]{$\alpha_{1}$};
\draw[->, thick] (0,0.2)--(0,0.7)node[anchor=west]{$\alpha_{3}$};
\draw[->, thick] (0,0.2)--(-0.3,-0.1) node[anchor=east]{$\alpha_{2}$};

\end{tikzpicture}

\caption{Découpage de $X_{*}(T)$ en $R_{b}$.}
\label{gl4vvv}
\end{center}
\end{figure}
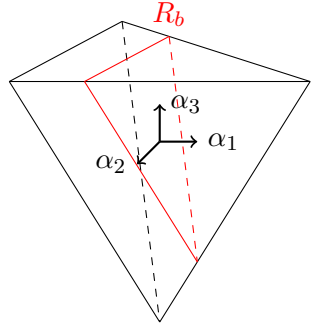

On note 
$$
P=\begin{bmatrix}*&&&\\ *&*&*&*\\ *&*&*&*\\ *&*&*&*
\end{bmatrix}.
$$
Soit $P=MN$ sa factorisation de Levi. On note 
$$
U=\begin{bmatrix}1&&&\\ \kp^{-m-b}&1&&\\ \kp^{-m-b}&&1&\\ \kp^{-m-b}&&&1
\end{bmatrix},
$$ 
alors $C(w)=UI^{M}wK/K$.

\begin{lem}\label{gl4retractaffine}

La rétraction 
$
f_{P}:V_{b}\to V_{b}\cap \xx^{M}
$ 
est une fibration en espaces affines. 
\end{lem}

\begin{proof}

On montre d'abord que la dimension de $\ku/\ku\cap \Ad(g)\kg(\co)$ est indépendante de $g$ pour $gK\in V_{b}\cap \xx^{M}$. Pour $g\in I^{M},\,w\in R_{b}$, on a
$$
\frac{\ku}{\ku\cap \Ad(gw)\kg(\co)}\cong \frac{\ku}{\ku\cap \Ad(w)\kg(\co)},$$
car  $I^{M}$ normalise $\ku$. Et la dimension de la dernière terme est 
$$
\sum_{i=2}^{4}[(w_{i}-w_{1})-(-m-w_{1})]=3m-w_{1}=3m-b.
$$
Donc les $\ku/\ku\cap \Ad(g)\kg(\co)$ s'organisent en un fibré vectoriel $\tilde{\ku}/\tilde{\ku}\cap \ckk$ sur $V_{b}\cap \xx^{M}$. La rétraction en question est un torseur sous ce fibré vectoriel, d'où le lemme.
\end{proof}

De plus, puisque $M=\gl_{1}\times \gl_{3}$, la projection vers le deuxième facteur donne un isomorphisme 
$$
V_{b}\cap \xx^{M}\cong \xx^{\gl_{3},(-b)}_{\geq -m}.
$$

Dans la suite, on distingue entre les deux types.

\subsubsection{Premier type}

On va couper $V_{b}$ en parties localement fermées telles que la restriction de la fibration 
$$
f_{P}:\xx_{\gamma}\cap V_{b}\to \xx^{M}_{\gamma}\cap V_{b}
$$ 
sur chaque partie est une fibration en espaces affines.

On fixe $c\in \br$ tel que 
$$
-m+n_{1}-1<c<-m+n_{1}.
$$ 
Pour $i=2,3,4$, on note 
\begin{eqnarray*}
\Sigma_{i}&=&\{w\in R_{b}\,\mid\,w_{i}<c,\,w_{j}>c,\,\forall j\neq i,\,j\neq 1\},\\
\Sigma^{*}_{i}&=&\{w\in R_{b}\,\mid\,w_{i}>c,\,w_{j}<c,\,\forall j\neq i,\,j\neq 1\},
\end{eqnarray*}
et on note
$$
\Sigma_{0}=
\begin{cases}
\{w\in R_{b}\,\mid\, w_{j}> c,\,\forall j\neq 1\},&\text{ si } c<\left[\frac{-b}{3}\right],\\
\{w\in R_{b}\,\mid\, w_{j}< c,\,\forall j\neq 1\},&\text{ si } c>\left[\frac{-b}{3}\right].
\end{cases}
$$
Alors $R_{b}=\Sigma_{0}\cup\bigcup_{i=2}^{4}(\Sigma_{i}\cup \Sigma_{i}^{*})$. Pour tout $c',\,c''\in \bz$, on note 
$$
\Sigma_{i,c'}=\{w\in \Sigma_{i}\,\mid\,w_{i}=c'\},\quad\Sigma^{*}_{i,c''}=\{w\in \Sigma^{*}_{i}\,\mid\,w_{i}=c''\}.
$$ 
Les figures \ref{gl43aaa} et \ref{gl43bbb} donnent les découpages dans les deux cas.

\begin{figure}
\begin{center}

\begin{tikzpicture}
\draw (-3.5, -2)--(3.5,2);
\draw (-3.5, 2)--(3.5,-2);
\draw (0, 3)--(0,-3.5);

\draw[red, dashed] (-2.5, 2.7) node[anchor=east]{$v_{1}=c$}--(1,-3.3);
\draw[red, dashed] (2.5, 2.7) node[anchor=east]{$v_{2}=c$}--(-1,-3.3);
\draw[red, dashed] (-3.7, 0.8) node[anchor=east]{$v_{3}=c$}--(3.7,0.8);

\draw[black]     (-1.962, 1.8) circle (2pt)
                 (1.962,1.8) circle (2pt)                 
                 (-2.54, 0.8) circle (2pt)
                 (2.54, 0.8) circle (2pt)
                 (0.577, -2.6) circle (2pt)                 
                 (-0.577, -2.6) circle (2pt);

\draw (-1.962, 1.8)-- node[anchor=south]{\color{blue} $\Sigma_{3,c'}$} (1.962,1.8);
\draw (-2.54, 0.8)-- node[anchor=east]{\color{blue}$\Sigma_{1,c'}$}(-0.577,-2.6);
\draw (2.54, 0.8)-- node[anchor=west]{\color{blue}$\Sigma_{2,c'}$}(0.577,-2.6);

\draw (1.962, 1.8)--node[anchor=west]{\color{blue}$\Sigma^{*}_{1,c''}$} (2.54, 0.8);
\draw (-1.962, 1.8)--node[anchor=east]{\color{blue}$\Sigma^{*}_{2,c''}$} (-2.54,0.8);
\draw (-0.577, -2.6)-- node[anchor=north]{\color{blue}$\Sigma^{*}_{3,c''}$}(0.577,-2.6);

\draw[black] (1.385, 0.8) circle (2pt)
             (-1.385, 0.8) circle (2pt)
             (0, -1.6) circle (2pt);
             
\draw[dashed, black] (-1.385,0.8)--(1.385,0.8);             
\draw[dashed, black] (-1.385,0.8)--(0,-1.6);
\draw[dashed, black] (0,-1.6)--(1.385,0.8);
\draw (0,0) node[anchor=north]{\color{blue}$\Sigma_{0}$};

\draw[->, thick](0,0)--(0.5,0) node[anchor=west]{$\alpha$};
\draw[->, thick](0,0)--(-0.25, 0.433) node[anchor=south]{$\beta$};
\end{tikzpicture}
\caption{Découpage de $R_{b}$ pour $c<\left[\frac{-b}{3}\right]$.}
\label{gl43aaa}
\end{center}
\end{figure}

\begin{figure}
\begin{center}

\begin{tikzpicture}
\draw (-3.5, -2)--(3.5,2);
\draw (-3.5, 2)--(3.5,-2);
\draw (0, 3)--(0,-3.5);

\draw[red, dashed] (-2.5, -2.7)--(1,3.3) node[anchor=west]{$v_{2}=c$};
\draw[red, dashed] (2.5, -2.7)--(-1,3.3) node[anchor=east]{$v_{1}=c$};
\draw[red, dashed] (-3.7, -0.8) node[anchor=east]{$v_{3}=c$}--(3.7,-0.8);

\draw[black]     (-1.962, -1.8) circle (2pt)
                 (1.962,-1.8) circle (2pt)                 
                 (-2.54, -0.8) circle (2pt)
                 (2.54, -0.8) circle (2pt)
                 (0.577, 2.6) circle (2pt)                 
                 (-0.577, 2.6) circle (2pt);

\draw (-1.962, -1.8)-- node[anchor=north]{\color{blue} $\Sigma^{*}_{3,c''}$} (1.962,-1.8);
\draw (-2.54, -0.8)-- node[anchor=east]{\color{blue}$\Sigma^{*}_{2,c''}$}(-0.577,2.6);
\draw (2.54, -0.8)-- node[anchor=west]{\color{blue}$\Sigma^{*}_{1,c''}$}(0.577,2.6);

\draw (1.962, -1.8)--node[anchor=west]{\color{blue}$\Sigma_{2,c'}$} (2.54, -0.8);
\draw (-1.962, -1.8)--node[anchor=east]{\color{blue}$\Sigma_{1,c'}$} (-2.54,-0.8);
\draw (-0.577, 2.6)-- node[anchor=south]{\color{blue}$\Sigma_{3,c'}$}(0.577,2.6);

\draw[black] (1.385, -0.8) circle (2pt)
             (-1.385, -0.8) circle (2pt)
             (0, 1.6) circle (2pt);
             
\draw[dashed, black] (-1.385,-0.8)--(1.385,-0.8);             
\draw[dashed, black] (-1.385,-0.8)--(0,1.6);
\draw[dashed, black] (0,1.6)--(1.385,-0.8);
\draw (0,0) node[anchor=north]{\color{blue}$\Sigma_{0}$};

\draw[->, thick](0,0)--(0.5,0) node[anchor=west]{$\alpha$};
\draw[->, thick](0,0)--(-0.25, 0.433) node[anchor=south]{$\beta$};
\end{tikzpicture}
\caption{Découpage de $R_{b}$ pour $c>\left[\frac{-b}{3}\right]$.}
\label{gl43bbb}
\end{center}
\end{figure}

\begin{lem}
La rétraction 
$$
f_{P}:\xx_{\gamma}\cap V_{b}\to \xx^{M}_{\gamma}\cap V_{b}
$$ 
induite une fibration en espaces affines
$$
f_{P}:\xx_{\gamma}\cap\bigcup_{w\in \Sigma_{0}}C(w)\to \xx^{M}_{\gamma}\cap\bigcup_{w\in \Sigma_{0}}C(w),
$$
et du même quand on remplace $\Sigma_{0}$ par $ \Sigma_{i,c'},\, \Sigma_{i,c''}^{*}$.
\end{lem}

\begin{proof}

Puisque on a déjà le lemme \ref{gl4retractaffine}, d'après la proposition \ref{techfam}, il suffit de montrer que la dimension de 
$$
\frac{\ad(\gamma)\ku}{\ad(\gamma)\ku\cap \Ad(g)\kg(\co)}
$$
est indépendante de $g$ pour $gK$ dans $\xx^{M}_{\gamma}\cap\bigcup_{w\in \Sigma_{0}}C(w)$. Même argument si on remplace $\Sigma_{0}$ par $ \Sigma_{i,c'},\, \Sigma_{i,c''}^{*}$.

Pour $w\in R_{b}$, on a $C(w)\cap \xx^{M}=I^{M}wK/K$. Pour $g\in I^{M}$, on a
$$
\frac{\ad(\gamma)\ku}{\ad(\gamma)\ku\cap \Ad(gw)\kg(\co)}\cong \frac{\ad(\gamma)\ku}{\ad(\gamma)\ku\cap \Ad(w)\kg(\co)},$$
car $I^{M}$ normalise $\ad(\gamma)\ku$. La dimension du dernier terme est
$$
\sum_{i=2}^{4}\max\{w_{i}+m-n_{1},\,0\}.
$$

\begin{enumerate}

\item Pour $c>\left[\frac{-b}{3}\right]$, $w\in \Sigma_{0}$, la dimension est $0$.

\item Pour $c<\left[\frac{-b}{3}\right]$, $w\in \Sigma_{0}$, la dimension est 
\begin{eqnarray*}
\sum_{j=2}^{4}(w_{j}+m-n_{1})=3(m-n_{1})-b,\end{eqnarray*}
constante sur $\Sigma_{0}$.

\item Pour $w\in\Sigma_{i,c'}$, la dimension est 
\begin{eqnarray*}
\sum_{j=2, \,j\neq i}^{4}(w_{j}+m-n_{1})&=&2(m-n_{1})-b-c',\end{eqnarray*}
constante sur $\Sigma_{i,c'}$.

\item Pour $w\in\Sigma^{*}_{i,c''}$, la dimension est 
$$
w_{i}+m-n_{1}=c''+m-n_{1},
$$ 
constante sur $\Sigma^{*}_{i,c''}$.
\end{enumerate}
\end{proof}

Puisque $M=\gl_{1}\times \gl_{3}$, d'après le théorème $\ref{famgl3}$ et le corollaire $\ref{gl3alt}$, les intersections
$$
\xx_{\gamma}^{M}\cap\bigcup_{w\in \Sigma_{0}}C(w),\quad \xx_{\gamma}^{M}\cap\bigcup_{w\in \Sigma_{i,c'}}C(w),\quad \xx_{\gamma}^{M}\cap\bigcup_{w\in \Sigma_{i,c''}^{*}}C(w).
$$
admettent des pavages en espaces affines, et on a vu comment les ordonner pour en déduire un pavage en espaces affines de $\xx_{\gamma}^{M}\cap V_{b}$. D'après le lemme précédent, les intersections 
$$
\xx_{\gamma}\cap\bigcup_{w\in \Sigma_{0}}C(w),\quad \xx_{\gamma}\cap\bigcup_{w\in \Sigma_{i,c'}}C(w),\quad \xx_{\gamma}\cap\bigcup_{w\in \Sigma_{i,c''}^{*}}C(w).
$$ 
admettent aussi des pavages en espaces affines, et en utilisant le même ordre que leurs analogues ci-dessus, on en déduit un pavage en espaces affines de $\xx_{\gamma}\cap V_{b}$. Le théorème est donc démontré pour ce type.

\subsubsection{Deuxième type}

On pave $V_{b}$ avec l'Iwahori $I'_{\ba}:=\Ad(\vep^{\ba})I'$, où 
$$
\ba=(n_{1}-n,n_{1}-n,0,0)\in \bn^{4}.
$$ 
On note $C_{\ba}(w)=V_{b}\cap I'_{\ba}wK/K$. Soit
$$
H=I'^{M}_{\ba}\cap \Ad(\diag(1,1,\vep^{m_{w}},\vep^{m_{w}}))K^{M},
$$ 
où $m_{w}=-m-w_{2}$. Alors les mêmes arguments que ceux dans la  démonstration de la proposition $\ref{tranche}$ et du lemme \ref{affine2} montrent que 

\begin{lem}\label{gl4pavnon}

On a un pavage en espaces affines
$$
V_{b}=\bigsqcup_{w\in R_{b}}C_{\ba}(w),
$$ 
où $C_{\ba}(w)=UHwK/K$, et la rétraction
$$
f_{P}:UHwK/K\to HwK^{M}/K^{M}
$$
est une fibration en espaces affines. 
\end{lem}

\begin{lem}\label{redgl4}

La rétraction 
$$
f_{P}:\xx_{\gamma}\cap C_{\ba}(w)\to \xx_{\gamma}^{M}\cap HwK^{M}/K^{M}
$$ 
est une fibration en espaces affines.

\end{lem}

\begin{proof}

D'après la proposition \ref{techfam} et le lemme \ref{gl4pavnon}, il suffit de montrer que la dimension $\ad(\gamma)\ku/\ad(\gamma)\ku\cap \Ad(g)\kg(\co)$ est indépendante de $g$ pour $gK\in HwK^{M}/K^{M}$. Puisque $H\subset I'^{M}_{\ba}$, il normalise $\ad(\gamma)\ku$. Donc l'énoncé se déroule du lemme \ref{normal}.\end{proof}

On note 
$$
R_{b1}=\{w\in R_{b}\,\mid\,w_{2}\geq -m+n_{1}-n\},\quad V_{b1}=\bigsqcup_{w\in R_{b1}}C_{\ba}(w). 
$$
C'est une sous-variété fermée de $V_{b}$. Voir le figure \ref{gl43rb} pour avoir une idée de $R_{b1}$. Dans les lemmes \ref{aa}, \ref{bb}, on va paver $(V_{b}\backslash V_{b1})\cap \xx_{\gamma}$ et $V_{b1}\cap \xx_{\gamma}$ en espaces affines. La réunion de ces deux pavages nous donnera un pavage en espaces affines de $V_{b}\cap \xx_{\gamma}$, ce qui terminera la démonstration pour le deuxième type.

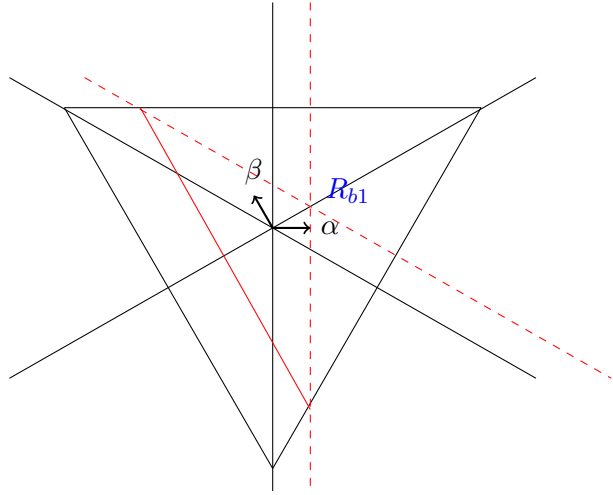
\begin{figure}
\begin{center}

\begin{tikzpicture}
\draw (-3.5, -2)--(3.5,2);
\draw (-3.5, 2)--(3.5,-2);
\draw (0, 3)--(0,-3.5);

\draw[black] (-2.77,1.6)--(2.77,1.6);             
\draw[black] (-2.77,1.6)--(0,-3.2);
\draw[black] (0,-3.2)--(2.77,1.6);
\draw (1,0.5) node{\color{blue}$R_{b1}$};

\draw[red] (0.5,-2.415)--(-1.76,1.6);
\draw[red, dashed] (-2.5, 2)--(4.5,-2);
\draw[red, dashed] (0.5, 3)--(0.5,-3.5);

\draw[->, thick](0,0)--(0.5,0) node[anchor=west]{$\alpha$};
\draw[->, thick](0,0)--(-0.25, 0.433) node[anchor=south]{$\beta$};
\end{tikzpicture}
\caption{$R_{b1}$ dans $R_{b}$.}
\label{gl43rb}
\end{center}
\end{figure}

\begin{lem}\label{aa}

La sous-variété ouverte $(V_{b}\backslash V_{b1})\cap \xx_{\gamma}$ de $V_{b}\cap \xx_{\gamma}$ admet un pavage en espaces affines. Plus précisément, pour $w\in R_{b}\backslash R_{b1}$, l'intersection $C_{\ba}(w)\cap \xx_{\gamma}$ est isomorphe à un espace affine.\end{lem}

\begin{proof}

D'après le lemme $\ref{redgl4}$, il suffit de montrer que $HwK^{M}/K^{M}\cap \xx_{\gamma}^{M}$ est isomorphe à un espace affine. Puisque 
$$
w_{2}-w_{j}\leq w_{2}-(-m)<n_{1}-n,\quad j=3,4,
$$
on a
$$
HwK^{M}/K^{M}=\begin{bmatrix}\co&&&\\ &\co&&\\ &\kp^{m_{w}}&\co&\co\\ &\kp^{m_{w}}&\kp&\co\end{bmatrix}wK^{M}/K^{M}.
$$
On note 
$$
H'=\begin{bmatrix}\co&&&\\ &\co&&\\ &&\co&\co\\ &&\kp&\co\end{bmatrix},\quad U'=\begin{bmatrix}1&&&\\ &1&&\\ &\kp^{m_{w}}&1&\\ &\kp^{m_{w}}&&1\end{bmatrix}.
$$
Ce sont des groupes de Lie et on a $HwK^{M}/K^{M}=U'H'wK^{M}/K^{M},\,\forall w\in R_{b}\backslash R_{b1}$. On note
$$
P'=\begin{bmatrix}*&&&\\ &*&&\\ &*&*&*\\ &*&*&*
\end{bmatrix},
$$
et $P'=M'N'$ sa factorisation de Levi. Parce que $H'$ normalise $\ku'$ et $\ad(\gamma)\ku'$, le corollaire $\ref{technormal}$ implique que la rétraction 
$$
f_{P'}: U'H'wK^{M}/K^{M}\cap \xx_{\gamma}^{M}\to H'wK^{M'}/K^{M'}\cap \xx_{\gamma}^{M'}
$$ 
est une fibration en espaces affines. De plus, l'intersection $H'wK^{M'}/K^{M'}\cap \xx_{\gamma}^{M'}$ est isomorphe à un espace affine car $M'=\gl_{1}\times \gl_{1}\times \gl_{2}$, d'où le lemme. 
\end{proof}

\begin{lem}\label{bb}
L'intersection $V_{b1}\cap \xx_{\gamma}$ admet un pavage en espaces affines.\end{lem}

\begin{proof}

En utilisant le lemme $\ref{redgl4}$, on va déduire un pavage de $V_{b1}\cap \xx_{\gamma}$ d'un pavage de $V_{b1}\cap \xx_{\gamma}^{M}$.

Pour $w\in R_{b1}$, on a 
$$
HwK^{M}/K^{M}=I_{\ba}'^{M}wK^{M}/K^{M}.
$$ 
La composition de la translation sur $\xx^{M}$ par $\vep^{-\ba}$ et la projection à la facteur $\gl_{3}$ de $M=\gl_{1}\times \gl_{3}$ donne un isomorphisme 
\begin{equation}\label{translategl3}
V_{b1}\cap \xx^{M}\cong \xx_{\geq -m}^{\gl_{3},(-b+n-n_{1})}.
\end{equation}
Avec cet isomorphisme, on peut translater le pavage pour $\gl_{3}$ à un pavage pour $V_{b1}\cap \xx_{\gamma}^{M}$. On fixe $c\in \br$ tel que
$$
n-m-1<c<n-m.
$$
On note $w_{2}'=w_{2}-(n_{1}-n)$ et $w_{j}'=w_{j},\,j=3,4$. Pour $i=2,3,4$, on note
\begin{eqnarray*}
\Sigma_{i}&=&\{w\in R_{b1}\,\mid\,w'_{i}< c,\,w'_{j}> c,\,\forall j\neq 1,i\},\\
\Sigma^{*}_{i}&=&\{w\in R_{b1}\,\mid\,w'_{i}>c,\,w'_{j}<c,\,\forall j\neq 1,i\},
\end{eqnarray*}
et on note
$$
\Sigma_{0}=
\begin{cases}
\{w\in R_{b1}\,\mid\, w'_{j}> c,\,\forall j\neq 1\},&\text{ si } c<\left[\frac{-b+n-n_{1}}{3}\right],\\
\{w\in R_{b1}\,\mid\, w'_{j}< c,\,\forall j\neq 1\},&\text{ si } c>\left[\frac{-b+n-n_{1}}{3}\right].
\end{cases}
$$
Alors $R_{b1}=\Sigma_{0}\cup \bigcup_{i=2}^{4}(\Sigma_{i}\cup \Sigma_{i}^{*})$. Pour tout $c',\,c''\in \bz$, on note 
$$
\Sigma_{i,c'}=\{w\in \Sigma_{i}\,\mid\,w'_{i}=c'\},\quad\Sigma^{*}_{i,c''}=\{w\in \Sigma^{*}_{i}\,\mid\,w'_{i}=c''\}.
$$ 
Ce découpage est analogue de celui indiqué dans les figures \ref{gl43aaa} et \ref{gl43bbb}.

\begin{lem}

La rétraction 
$$
f_{P}:\xx_{\gamma}\cap V_{b1}\to \xx_{\gamma}^{M}\cap V_{b1}
$$ 
induite une fibration en espaces affines 
$$
f_{P}:\xx_{\gamma}\cap \bigsqcup_{w\in \Sigma_{0}}C_{\ba}(w)\to \xx_{\gamma}^{M}\cap \bigsqcup_{w\in \Sigma_{0}}C_{\ba}(w),
$$
et du même si on remplace $\Sigma_{0}$ par $\Sigma_{i,c'},\,\Sigma^{*}_{i,c''}$.

\end{lem}

\begin{proof}

Puisque on a déjà le lemme \ref{gl4retractaffine}, d'après la proposition \ref{techfam}, il suffit de montrer que la dimension de 
$$
\frac{\ad(\gamma)\ku}{\ad(\gamma)\ku\cap \Ad(g)\kg(\co)}
$$
est indépendante de $g$ pour $gK$ dans $\xx_{\gamma}^{M}\cap \bigsqcup_{w\in \Sigma_{0}}C_{\ba}(w)$. Même argument si on remplace $\Sigma_{0}$ par $\Sigma_{i,c'},\,\Sigma^{*}_{i,c''}$.

Pour $w\in R_{b1}$, on a $C_{\ba}(w)\cap \xx^{M}=I_{\ba}'^{M}wK/K$. Pour $g\in I_{\ba}'^{M}$, on a
$$
\frac{\ad(\gamma)\ku}{\ad(\gamma)\ku\cap \Ad(gw)\kg(\co)}\cong \frac{\ad(\gamma)\ku}{\ad(\gamma)\ku\cap \Ad(w)\kg(\co)},$$
car $I_{\ba}'^{M}$ normalise $\ad(\gamma)\ku$. La dimension du dernier terme est
$$
\max\{w_{2}+m-n_{1},\,0\}+\sum_{i=3}^{4}\max\{w_{i}+m-n,\,0\}.
$$

\begin{enumerate}

\item Pour $c<\left[\frac{-b+n-n_{1}}{3}\right]$, soit $w\in \Sigma_{0}$, la dimension est 
\begin{eqnarray*}
&&(w_{2}+m-n_{1})+\sum_{j=3}^{4}(w_{j}+m-n)\\ &&=3m-b-n_{1}-2n,
\end{eqnarray*}
donc constante sur $\Sigma_{0}$.

\item Pour $c>\left[\frac{-b+n-n_{1}}{3}\right]$, soit $w\in \Sigma_{0}$, la dimension est $0$.

\item Pour $w\in \Sigma_{i,c'}$, la dimension est 
\begin{eqnarray*}
\begin{cases}2(m-n)-b-c', & \text{si } i=2;\\ 2m-n_{1}-n-b-c', & \text{si } i=3,4; 
\end{cases}
\end{eqnarray*}
donc constante sur $\Sigma_{i,c'}$.

\item Pour $w\in \Sigma_{i,c''}^{*}$, la dimension est 
\begin{eqnarray*}
\begin{cases} c''+m-n_{1}, & \text{si } i=2; \\c''+m-n, & \text{si } i=3,4;\end{cases}
\end{eqnarray*}
donc constante sur $\Sigma_{i,c''}^{*}$.
\end{enumerate}\end{proof}

Par l'isomorphisme $(\ref{translategl3})$, les intersections 
$$
\xx^{M}_{\gamma}\cap \bigsqcup_{w\in \Sigma_{0}}C_{\ba}(w),\quad\xx^{M}_{\gamma}\cap \bigsqcup_{w\in \Sigma_{i,c'}}C_{\ba}(w),\quad 
\xx^{M}_{\gamma}\cap \bigsqcup_{w\in \Sigma^{*}_{i,c''}}C_{\ba}(w).
$$
admettent des pavages en espaces affines d'après le théorème $\ref{famgl3}$ et le corollaire $\ref{gl3alt}$. Comme expliqué dans la construction de \S2.4, on peut les ordonner pour en déduire un pavage en espaces affines de $\xx_{\gamma}^{M}\cap V_{b1}$. D'après le lemme précédent, les intersections 
$$
\xx_{\gamma}\cap \bigsqcup_{w\in \Sigma_{0}}C_{\ba}(w),\quad\xx_{\gamma}\cap \bigsqcup_{w\in \Sigma_{i,c'}}C_{\ba}(w),\quad 
\xx_{\gamma}\cap \bigsqcup_{w\in \Sigma^{*}_{i,c''}}C_{\ba}(w).
$$
admettent aussi des pavages en espaces affines. En utilisant le même ordre que leurs analogues ci-dessus, on en déduit un pavage de $\xx_{\gamma}\cap V_{b1}$ en espaces affines.
\end{proof}

\begin{rem}
\begin{enumerate}

\item La méthode que l'on a développé pour paver les fibres de Springer affines pour $\gl_{4}$ peut être généralisée aux groupes classiques de rang $2$ et $3$ sans grandes difficultés. 


\item  La difficulté principale pour généraliser cette méthode à $\gl_{d},\,d\geq 5$ est due au fait que les intersections $\xx_{\gamma}\cap S_{\varpi}^{c}(v)$ ne sont pas pure si $\varpi$ n'est pas conjugué à $\varpi_{1}$ ou $\varpi_{d-1}$ sous l'action de $W$. 

\end{enumerate}

\end{rem}

\section*{Appendice: Forme minimale d'un élément semi-simple régulier non-ramifié}

Soit $G=\gl_{d}$, $T$ le tore maximal des matrices diagonales. 

\begin{defn}
L'élément régulier $\gamma\in \kt(F)$ est dit \textit{en forme minimale} s'il satisfait à la condition 
$$
\val(\alpha_{i,j}(\gamma))=\min_{i\leq l\leq j-1}\{\val(\alpha_{l}(\gamma))\},\quad \forall i,j=1,\cdots,d,\, i<j.
$$ 
Dans ce cas, on dit que sa valuation radicielle est le $(d-1)$-uplet $(\val(\alpha_{i}(\gamma))_{i=1}^{d-1})$.

\end{defn}

\begin{prop}\label{root}
Tout élément $\gamma\in \kt(F)$ est conjugué sous l'action du groupe de Weyl à au moins un élément en forme minimale.
\end{prop}

\begin{proof}
On va montrer la proposition par récurrence. Pour $G=\gl_{2}$, le résultat est évident. On suppose que pour $G=\gl_{d'},\,d'<d,$ la proposition est démontrée.

Soit $\gamma=\diag(\gamma_{1},\cdots,\gamma_{d})$, soit $n=\max\{\val(\alpha_{i,j}(\gamma))\mid i\neq j\}$. Soit $\gamma'_{1}=\diag(\gamma_{\tau(1)},\cdots,\gamma_{\tau(a)})$ l'une des sous-matrices équivaluées de valuation $n$ de $\gamma$ qui est de taille maximale. On note $\gamma_{2}'=\diag(\gamma_{\tau(a+1)},\cdots,\gamma_{\tau(d)})$, et $\gamma'=\diag(\gamma_{1}',\gamma_{2}')$.

\begin{lem}
Pour $a+1\leq i\leq d$ fixé et $1\leq j\leq a$, les valuations $\val(\alpha_{i,j}(\gamma'))=\val(\alpha_{j,i}(\gamma'))$ sont toutes les mêmes et strictement plus petites que $n$.
\end{lem}

\begin{proof}
S'il existe $a+1\leq i_{0}\leq d,\, 1\leq j_{0}\leq a$ tel que $\val(\alpha_{i_{0},j_{0}}(\gamma'))=n$, alors pour tout $1\leq j\leq a$, les inégalités 
$$
n\geq\val(\alpha_{i_{0},j}(\gamma'))\geq \min\{\val(\alpha_{i_{0},j_{0}}(\gamma')),\val(\alpha_{j_{0},j}(\gamma'))\}=n
$$
entrainent que $\val(\alpha_{i_{0},j}(\gamma'))=n$, i.e. que la matrice $\diag(\gamma_{\tau(1)},\cdots,\gamma_{\tau_{(a)}},\gamma'_{i_{0},i_{0}})$ est équivaluée de valuation $n$, contradiction à l'hypothèse que $\gamma_{1}'$ est de taille maximale. Donc $\val(\alpha_{i,j}(\gamma'))<n$, pour tout $a+1\leq i\leq d, \,1\leq j\leq a$.

Par conséquent, pour tout $1\leq k,l\leq a$, on a 
$$
\val(\alpha_{i,k}(\gamma'))=\min\{\val(\alpha_{i,l}(\gamma')),\,\val(\alpha_{l,k}(\gamma'))=n\}=\val(\alpha_{i,l}(\gamma')).
$$  
\end{proof}

Donc on peut ``contracter'' $\gamma'_{1}$ en un élément $\delta_{1}\in F$ tel que $\val(\delta_{1})=n$ et les valuations radicielles ne changent pas à l'extérieure de $\gamma'_{1}$. On note $\delta=(\delta_{1}, \gamma'_{2})$. On observe que la matrice $\gamma'_{1}$ est toujours de taille strictement plus grande que $1$, donc $\delta$ est de taille strictement plus petite que $d-1$. Par l'hypothèse de récurrence, $\delta$ admet une forme minimale $\tau'(\delta)$. En rempla\c cant l'élément $\delta_{1}$ dans $\tau'(\delta)$ par $\gamma'_{1}$, on trouve une conjugaison de $\gamma$ en forme minimale.\end{proof}

\begin{rem}

En général, un élément $\gamma$ peut être conjugué à plusieurs éléments en forme minimale.\end{rem}

\end{document}